\documentclass{article}%
\usepackage{amsmath}
\usepackage{amsfonts}
\usepackage{amssymb}
\usepackage{amsthm}
\usepackage{graphicx}%
\newtheorem{theorem}{Theorem}[section]
\newtheorem{corollary}[theorem]{Corollary}
\newtheorem{lemma}[theorem]{Lemma}
\newtheorem{proposition}[theorem]{Proposition}
\newtheorem{remark}[theorem]{Remark}

\theoremstyle{definition}
\newtheorem{definition}[theorem]{Definition}

\theoremstyle{remark}
\newtheorem{example}[theorem]{Example}

\numberwithin{equation}{section}

\def\squareforqed{\hbox{\rlap{$\sqcap$}$\sqcup$}}
\def\qed{\ifmmode\else\unskip\quad\fi\squareforqed}
\def\smartqed{\def\qed{\ifmmode\squareforqed\else{\unskip\nobreak\hfil
\penalty50\hskip1em\null\nobreak\hfil\squareforqed
\parfillskip=0pt\finalhyphendemerits=0\endgraf}\fi}}

\smartqed

\begin{document}
\title{Analytic approximation of transmutation operators and related systems of functions}

\author{Vladislav V. Kravchenko and Sergii M. Torba\\{\small Departamento de Matem\'{a}ticas, CINVESTAV del IPN, Unidad
Quer\'{e}taro, }\\{\small Libramiento Norponiente No. 2000, Fracc. Real de Juriquilla,
Quer\'{e}taro, Qro. C.P. 76230 MEXICO}\\{\small e-mail: vkravchenko@math.cinvestav.edu.mx,
storba@math.cinvestav.edu.mx \thanks{Research was supported by CONACYT, Mexico
via the projects 166141 and 222478.}}}

\maketitle

\begin{abstract}
In \cite{KT AnalyticApprox} a method for approximate solution of
Sturm-Liouville equations and related spectral problems was presented based on
the construction of the Delsarte transmutation operators. The problem of
numerical approximation of solutions and eigendata was reduced to
approximation of a primitive of the potential by a finite linear combination
of certain specially constructed functions obtained from the generalized wave
polynomials introduced in \cite{KKTT}, \cite{KT Transmut}. The method allows
one to compute both lower and higher eigendata with an extreme accuracy.

Since the solution of the approximation problem is the main step in the
application of the method, the properties of the system of functions involved
are of primary interest. In \cite{KT AnalyticApprox} two basic properties were
established: the completeness in appropriate functional spaces and the linear
independence. In this paper we present a considerably more complete study of
the systems of functions. We establish their relation with another linear
differential second-order equation, find out certain operations (in a sense,
generalized derivatives and antiderivatives) which allow us to generate the
next such function from a previous one. We obtain the uniqueness of the
coefficients of expansions in terms of such functions and a corresponding
generalized Taylor theorem, as well as formulas for exact expansion
coefficients involving the operations mentioned above. We also construct the
invertible integral operators transforming powers of the independent variable
into the functions under consideration and establish their commutation
relations with differential operators. We present some error bounds for the solution of the approximation problem depending on the smoothness of the potential and show that these error bounds are close to optimal in order.
Also, we provide a rigorous  justification of the alternative formulation of the proposed method allowing one to make use of the known initial values of the solutions at the left endpoint of the spectral problem.
%\keywords{Transmutation operators \and Schr\"{o}dinger equation \and Approximate solution \and Sturm-Liouville problem \and Spectral parameter power series \and Formal powers}
%\subclass{MSC 34A25 \and 34A45 \and 34B24 \and 34L16 \and 34L40 \and 35L05 \and 35Q40  \and 41A25  \and 41A30  \and 41A58 \and 45L05  \and 47G10  \and 47N20  \and 47N40  \and 65L05  \and 65L15  }
\end{abstract}

\section{Introduction}

%Solution of Sturm-Liouville equations and of a wide range of related direct
%and inverse spectral problems is at the core of modern mathematical physics
%and its numerous applications. Since the work of J. Fourier on the theory of
%the heat and his method of separation of variables the properties and
%methods for solving different kinds of Sturm-Liouville spectral problems
%were studied in thousands of publications.

One of the important mathematical
tools for studying problems related to Sturm-Liouville equations was
introduced in 1938 by J. Delsarte \cite{Delsarte} and called \cite{DelsarteLions1956} the transmutation operator. It relates two linear
differential operators and allows one to transform a more complicated
equation into a simpler one. Nowadays the transmutation operator is widely
used in the theory of linear differential equations (see, e.g., \cite{BegehrGilbert}, \cite{Carroll}, \cite{Levitan}, \cite{Marchenko},
\cite{Sitnik}, \cite{Trimeche}). Very often in the literature the transmutation
operators are called the transformation operators. It is well known that
under certain regularity conditions the transmutation operator transmuting
the operator $A=-\frac{d^{2}}{dx^{2}}+q(x)$ into $B=-\frac{d^{2}}{dx^{2}}$
can be realized in the form of a Volterra integral operator with good properties. Its integral kernel
can be obtained as a solution of a certain Goursat problem for the
Klein-Gordon equation with a variable coefficient. In spite of their
attractive properties and importance there exist very few examples of the
transmutation kernels available in a closed form (see \cite{KrT2012}).

Several recent results
concerning the transmutation operators made it possible in \cite{KT
AnalyticApprox}, \cite{KMoTo} to convert the transmutation operators from a purely theoretical tool into an efficient practical method for solving Sturm-Liouville equations and related spectral problems. The method is called the analytic approximation of transmutation operators. It is based on the result from \cite{KKTT} where a new complete system of solutions for the Klein-Gordon equation was constructed. The functions of that system are called generalized wave polynomials. Since the integral kernel of the transmutation operator is a solution of that Klein-Gordon equation, it is possible to approximate the integral kernel as well. It was shown in \cite{KT AnalyticApprox} how the problem of approximation of the integral kernel reduces to the solution of two one-dimensional problems of approximation of the pair of functions $g_{1}(x)=\frac{h}{2}+\frac{1}{4}\int_{0}^{x}q(s)ds$ and $g_{2}(x)=\frac{1}{4}\int_{0}^{x}q(s)ds$ in terms of specially constructed families of functions $\{\mathbf{c}_n\}$ and $\{\mathbf{s}_n\}$, appearing as traces of the generalized wave polynomials on the line $x=t$.
With respect to the variable of integration the approximate kernel
results to be a polynomial. This is especially convenient since to obtain solutions of the Schr\"{o}dinger
equation the transmutation operator is applied to the functions $\sin \sqrt{\lambda }t$ and $\cos \sqrt{\lambda }t$, solutions of the simplest such equation $Bv=\lambda v$, and thus all the involved integrals can be calculated
explicitly. One of the advantages of the method is that due to the independence of the integral kernel of the spectral parameter the error of the computed eigendata does not increase for higher
eigenvalues. One can compute, e.g., the $1000$th eigenvalue and
eigenfunction with roughly the same accuracy as the first ones.

In \cite{KT AnalyticApprox} we presented a rigorous
justification of the method, however several important questions remained unanswered. It was observed in various numerical examples that the method allows one to obtain highly accurate eigendata and demonstrates exponential convergence with respect to the number of the functions $\{\mathbf{c}_n\}$ and $\{\mathbf{s}_n\}$ used while we only proved that the analytic approximations converge without any convergence rate estimates. Also the functionality of the method was proved for the symmetric segment $[-b,b]$ and in such form the method did not allow one to make use of the known initial values of the solutions at $x=0$.

The present paper is dedicated to further study of the analytic approximations of transmutation operators and of the systems of functions  $\{\mathbf{c}_n\}$ and $\{\mathbf{s}_n\}$ involved. In particular, we show that the systems of functions $\{\mathbf{c}_n\}$ and $\{\mathbf{s}_n\}$ are closely related with another linear differential second-order equation, find out certain operations (in a sense, generalized derivatives) which allow us to present Taylor-type formulas in terms of the linear combinations of the functions $\{\mathbf{c}_n\}$ and $\{\mathbf{s}_n\}$. We also construct the
invertible integral operators transforming powers of the independent variable
into the functions $\{\mathbf{c}_n\}$ and $\{\mathbf{s}_n\}$ and establish their commutation relations with differential operators. We prove that the convergence rate of the analytic approximations method depends on the smoothness of the potential and provide the justification of the alternative formulation of the method allowing one to make use of the known initial values of the solutions at an endpoint of the spectral problem.

The paper is structured as follows. In Section \ref{Sect 2} we introduce some necessary notations, definitions and properties concerning special systems of functions called formal powers, generalized wave polynomials and their traces $\{\mathbf{c}_n\}$ and $\{\mathbf{s}_n\}$. We present the definition of  the transmutation operators, recall some properties and briefly outline the method of analytic approximation of the transmutation operators. Also we present new relations for the integral kernel of the transmutation operator on the characteristics $x=t$ and $x=-t$. We would like to mention that these  relations \eqref{K1(x,x)} and \eqref{K1(x,-x)} already found applications to a posteriori accuracy verification in \cite{KNT}. In Section \ref{Sect 3} we show that the functions $\{\mathbf{c}_n\}$ and $\{\mathbf{s}_n\}$ are closely related to the formal powers for the equation $y'' - q(x)y = 2\lambda y'$. In Section \ref{Sect 4} we introduce the generalized derivatives $\gamma_1$ and $\gamma_2$ acting as $\gamma_2\gamma_1 \mathbf{c}_n = n\mathbf{s}_{n-1}$ and $\gamma_2\gamma_1 \mathbf{s}_n = n\mathbf{c}_{n-1}$ (the second formula is valid for $n\ge 2$ only). Also we study Taylor-type formulas in terms of the functions $\{\mathbf{c}_n\}$ and $\{\mathbf{s}_n\}$. We present the formulas for the coefficients  and prove the corresponding Taylor-type theorem with the Peano form of the remainder term. In Section \ref{Sect 5} we obtain a new representation for the preimage of the integral kernel of the transmutation operator. Also we prove that the even coefficients of the generalized Taylor series for the two functions related to the integral kernel coincide. In Section \ref{Sect 6} we study the Goursat-to-Goursat transmutation operators $G_1$ and $G_2$ mapping the powers of the independent variable to the functions $\mathbf{c}_n$ and $\mathbf{s}_n$ respectively. We prove some commutation relations involving these operators, derivatives and generalized derivatives $\gamma_2\gamma_1$. We show that the existence of certain number of the generalized derivatives $(\gamma_2\gamma_1)^j$ of a function implies a certain order differentiability of the preimage under the action of either operator $G_1$ or operator $G_2$ and vice versa. In Section \ref{Sect ConvEstimates} we show that there exists a close relation between the smoothness of the potential $q$ and the convergence rate of the analytic approximation of the transmutation operator, corresponding direct and inverse theorems are proved. In Section \ref{Sect 8} we present alternative proofs of the main theorems of the analytic approximation method. The proofs are based on the well-posedness of the Goursat problem and does not involve neither inverse operators nor pseudoanalytic function theory. Also the proofs justify the applicability of the analytic approximation method in the restricted setting, when the equation in considered on the half-segment $[0,b]$ only and no continuation of the potential onto the whole segment $[-b,b]$ is used.

\section{Transmutation operators and systems of functions $\left\{
\mathbf{c}_{n}\right\}  $ and $\left\{  \mathbf{s}_{n}\right\}  $}
\label{Sect 2}

\subsection{Formal powers}

Let $f\in C[a,b]$ be a complex valued function and $f(x)\neq0$ for any
$x\in\lbrack a,b]$. The interval $(a,b)$ is assumed being finite. Consider two
sequences of recursive integrals%
\begin{equation*}
X^{(0)}(x)\equiv1,\qquad X^{(n)}(x)=n\int_{x_{0}}^{x}X^{(n-1)}(s)\left(
f^{2}(s)\right)  ^{(-1)^{n}}\,\mathrm{d}s,\qquad x_{0}\in\lbrack a,b],\quad
n=1,2,\ldots%\label{Xn}%
\end{equation*}
and
\begin{equation*}
\widetilde{X}^{(0)}\equiv1,\qquad\widetilde{X}^{(n)}(x)=n\int_{x_{0}}%
^{x}\widetilde{X}^{(n-1)}(s)\left(  f^{2}(s)\right)  ^{(-1)^{n-1}}%
\,\mathrm{d}s,\qquad x_{0}\in\lbrack a,b],\quad n=1,2,\ldots. %\label{Xtiln}%
\end{equation*}

Define two families of functions $\left\{  \varphi_{k}\right\}  _{k=0}%
^{\infty}$ and $\left\{  \psi_{k}\right\}  _{k=0}^{\infty}$ constructed
according to the rules
\begin{equation}
\varphi_{k}(x)=%
\begin{cases}
f(x)X^{(k)}(x), & k\text{\ odd},\\
f(x)\widetilde{X}^{(k)}(x), & k\text{\ even},
\end{cases}
\label{phik}%
\end{equation}
and
\begin{equation}
\psi_{k}(x)=%
\begin{cases}
\dfrac{\widetilde{X}^{(k)}(x)}{f(x)}, & k\text{\ odd,}\\
\dfrac{X^{(k)}(x)}{f(x)}, & k\text{\ even}.
\end{cases}
\label{psik}%
\end{equation}

\subsection{SPPS representations}

The following result obtained in \cite{KrCV08} (for additional details and
simpler proof see \cite{APFT} and \cite{KrPorter2010}) establishes the
relation of the system of functions $\left\{  \varphi_{k}\right\}
_{k=0}^{\infty}$ and $\left\{  \psi_{k}\right\}  _{k=0}^{\infty}$ to the
Sturm-Liouville equation.

\begin{theorem}
\label{ThGenSolSturmLiouville} Let $q$ be a continuous complex valued function
of an independent real variable $x\in\lbrack a,b]$ and $\lambda$ be an
arbitrary complex number. Let $f$ be a solution of the equation
\begin{equation}
f^{\prime\prime}-qf=0 \label{SLhom}%
\end{equation}
on $(a,b)$ such that $f\in C^{2}(a,b)\cap C^{1}[a,b]$ and $f(x)\neq0$\ for any
$x\in\lbrack a,b]$ (it always exists, see Remark \ref{RemarkNonVanish}). Then the
general solution $y\in C^{2}(a,b)\cap C^{1}[a,b]$ of the equation
\begin{equation}
y^{\prime\prime}-qy=\lambda y \label{SLlambda}%
\end{equation}
on $(a,b)$ has the form $y=c_{1}y_{1}+c_{2}y_{2}$ where $c_{1}$ and $c_{2}$
are arbitrary complex constants,
\begin{equation}
y_{1}=\sum_{k=0}^{\infty}\frac{\lambda^{k}}{(2k)!}\varphi_{2k}\qquad
\text{and}\qquad y_{2}=\sum_{k=0}^{\infty}\frac{\lambda^{k}}{(2k+1)!}%
\varphi_{2k+1} \label{u1u2}%
\end{equation}
and both series converge uniformly on $[a,b]$ together with the series of the
first derivatives which have the form%
\begin{equation}\label{du1du2}
y_{1}^{\prime}=f^{\prime}+\sum_{k=1}^{\infty}\frac{\lambda^{k}}{(2k)!}\left(
\frac{f^{\prime}}{f}\varphi_{2k}+2k\,\psi_{2k-1}\right)  \quad\text{and}\quad
y_{2}^{\prime}=\sum_{k=0}^{\infty}\frac{\lambda^{k}}{(2k+1)!}\left(
\frac{f^{\prime}}{f}\varphi_{2k+1}+\left(  2k+1\right)  \psi_{2k}\right)  .
\end{equation}
The series of the second derivatives converge uniformly on any segment
$[a_{1},b_{1}]\subset(a,b)$.
\end{theorem}

Representations \eqref{u1u2} and \eqref{du1du2} are called SPPS
 (Spectral Parameter Power Series) representations. They have been implemented
for solving a variety of spectral and scattering problems related to
Sturm-Liouville equations. The first work using Theorem
\ref{ThGenSolSturmLiouville} for numerical solution was \cite{KrPorter2010}
and later on the SPPS method was used in a number of publications (see
\cite{CKOR}, \cite{CKT2013}, \cite{CKT2015}, \cite{ErbeMertPeterson2012}, \cite{KKB}, \cite{KKRosu},
\cite{KiraRosu2010}, \cite{KT Obzor}, \cite{KTV} and references therein).

\begin{remark}
\label{RemInitialValues}It is easy to see that by definition the solutions
$y_{1}$ and $y_{2}$ from \eqref{u1u2} satisfy the following initial
conditions
\begin{align*}
y_{1}(x_{0})  &  =f(x_{0}), & y_{1}^{\prime}(x_{0})  &  =f^{\prime}(x_{0}),\\
y_{2}(x_{0})  &  =0, & y_{2}^{\prime}(x_{0})  &  =1/f(x_{0}).
\end{align*}
\end{remark}

\begin{remark}
\label{RemarkNonVanish} It is worth mentioning that in the regular case the
existence and construction of the required $f$ presents no difficulty. Indeed,
let $q$ be real valued and continuous on $[a,b]$. Then \eqref{SLhom} possesses
two linearly independent real-valued solutions $f_{1}$ and $f_{2}$ whose zeros
alternate. Thus, one may choose $f=f_{1}+if_{2}$. Moreover, for the
construction of $f_{1}$ and $f_{2}$ in fact the same SPPS method may be used
\cite{KrPorter2010}. In the case of complex-valued coefficients the existence
of a non-vanishing solution was shown in \cite[Remark 5]{KrPorter2010}, see also \cite{Camporesi et al 2011}.
\end{remark}

\subsection{Transmutation operators}\label{Subsection Transmutation Operators}

Let $E$ be a linear topological space and $E_{1}$ its linear subspace (not
necessarily closed). Let $A$ and $B$ be linear operators: $E_{1}\rightarrow E$.

\begin{definition}
\label{DefTransmut} A linear invertible operator $T$ defined on the whole $E$
such that $E_{1}$ is invariant under the action of $T$ is called a
transmutation operator for the pair of operators $A$ and $B$ if it fulfills
the following two conditions.

\begin{enumerate}
\item Both the operator $T$ and its inverse $T^{-1}$ are continuous in $E$;

\item The following operator equality is valid
\begin{equation}
AT=TB \label{ATTB}%
\end{equation}
or which is the same
\[
A=TBT^{-1}.
\]

\end{enumerate}
\end{definition}

Our main interest concerns the situation when $A=-\frac{d^{2}}{dx^{2}}+q(x)$,
$B=-\frac{d^{2}}{dx^{2}}$,  and $q$ is a continuous complex-valued
function. Consider the space $E=C[-b,b]$ with $b$ being a positive number. In
\cite{CKT} and \cite{KrT2012} a parametrized family of transmutation operators
for $A$ and $B$ was studied. Operators of this family can be realized in the
form of a Volterra integral operator with a kernel associated to a value of
the complex parameter $h$,
\begin{equation}
\mathbf{T}u(x)=u(x)+\int_{-x}^{x}\mathbf{K}(x,t;h)u(t)dt. \label{Tmain}%
\end{equation}
It is convenient to interpret the parameter $h$ as follows. Let $f$ be a
solution of (\ref{SLhom}) satisfying the initial conditions
\begin{equation*}
f(0)=1\qquad\text{and}\qquad f^{\prime}(0)=h. %\label{initial cond f}%
\end{equation*}
Then there exists a unique operator of the form (\ref{Tmain}) satisfying
(\ref{ATTB}) and such that $\mathbf{T}\left[  1\right]  =f$. We will denote it
by $\mathbf{T}_{f}$ and its kernel by $\mathbf{K}(x,t)$. Thus, the operator
$\mathbf{T}_{f}$ has the form
\begin{equation}
\mathbf{T}_{f}u(x)=u(x)+\int_{-x}^{x}\mathbf{K}(x,t)u(t)dt. \label{Tf}%
\end{equation}
The kernel can be defined as $\mathbf{K}(x,t)=\mathbf{H}\big(\frac{x+t}%
{2},\frac{x-t}{2}\big)$, $|t|\leq|x|\leq b$, $\mathbf{H}$ being the unique
solution of the Goursat problem
\begin{equation}
\frac{\partial^{2}\mathbf{H}(u,v)}{\partial u\,\partial v}=q(u+v)\mathbf{H}%
(u,v), \label{GoursatTh1}%
\end{equation}%
\begin{equation}
\mathbf{H}(u,0)=\frac{h}{2}+\frac{1}{2}\int_{0}^{u}q(s)\,ds,\qquad
\mathbf{H}(0,v)=\frac{h}{2} \label{GoursatTh2}%
\end{equation}
where $h:=f^{\prime}(0)$.

If the potential $q$ is continuously differentiable, the kernel $\mathbf{K}$
itself is a solution of the Goursat problem
\begin{equation}
\left(  \frac{\partial^{2}}{\partial x^{2}}-q(x)\right)  \mathbf{K}%
(x,t)=\frac{\partial^{2}}{\partial t^{2}}\mathbf{K}(x,t), \label{GoursatKh1}%
\end{equation}%
\begin{equation}
\mathbf{K}(x,x)=\frac{h}{2}+\frac{1}{2}\int_{0}^{x}q(s)\,ds,\qquad
\mathbf{K}(x,-x)=\frac{h}{2}. \label{GoursatKh2}%
\end{equation}
If the potential $q$ is $n$ times continuously differentiable, the kernel
$\mathbf{K}(x,t)$ is $n+1$ times continuously differentiable with respect to
both independent variables.

The following theorem states that the operators $\mathbf{T}_{f}$ are indeed
transmutations in the sense of Definition \ref{DefTransmut}.

\begin{theorem}
[\cite{KT Transmut}]\label{Th Transmutation} Let $q\in C[-b,b]$. Then the
operator $\mathbf{T}_{f}$ defined by \eqref{Tf} satisfies the equality
\begin{equation}
\left(  -\frac{d^{2}}{dx^{2}}+q(x)\right)  \mathbf{T}_{f}[u]=\mathbf{T}%
_{f}\left[  -\frac{d^{2}}{dx^{2}}(u)\right]  \label{transmutationT}%
\end{equation}
for any $u\in C^{2}[-b,b]$.
\end{theorem}

\begin{remark}
\label{RemTh}$\mathbf{T}_{f}$ maps a solution $v$ of the equation
$v^{\prime\prime}+\omega^{2}v=0$, where $\omega$ is a complex number, into a
solution $u$ of the equation
\begin{equation}
u^{\prime\prime}-q(x)u+\omega^{2}u=0 \label{SLomega2}%
\end{equation}
with the following correspondence of the initial values $u(0)=v(0)$,
$u^{\prime}(0)=v^{\prime}(0)+hv(0)$.
\end{remark}

The integral kernel $\mathbf{K}(x,t)$ satisfies in the region $|t|\le |x|\le b$ the following integral equation (see \cite[(3.3) and (3.4)]{KT Transmut} and \cite[\S 2]{Marchenko}),
\begin{equation}
\label{IntEq for K}\mathbf{K}(x,t) = \frac{h}2 + \frac{1}2\int_{0}^{\frac{x+t}2}
q(s)\,ds + \int_{0}^{\frac{x+t}2} \int_{0}^{\frac{x-t}2}q(\alpha
+\beta)\mathbf{K}(\alpha+\beta, \alpha-\beta)\,d\beta\,d\alpha.
\end{equation}
Differentiating \eqref{IntEq for K} with respect to $t$ we obtain, c.f., \cite[(3.11)]{KNT}, that
\begin{equation}
\label{IntEq for dtK}%
\mathbf{K}_{2}(x,t)  =\frac{1}4 q\left(  \frac{x+t}2\right)  + \frac{1}2\int_{\frac{x+t}2}^{x}
q(z)\mathbf{K}(z, x+t-z)\,dz - \frac{1}2\int_{\frac{x-t}2}^{x} q(z)\mathbf{K}(z, z-(x-t))\,dz,
\end{equation}
here and below by $\mathbf{K}_1(x,t)$ and $\mathbf{K}_2(x,t)$ we
denote the partial derivatives with respect to the first and the second
variable respectively.

Let us introduce the following useful notation
\[
Q(x):=\int_{0}^{x}q(s)ds.
\]
\begin{proposition}
The following relations for the partial derivatives of the transmutation
kernel are valid on the line $t=x$,
\begin{equation}
\mathbf{K}_{1}(x,x)=\frac{1}{4}\left(  q(x)+hQ(x)+\frac{Q^{2}(x)}{2}\right)
\label{K1(x,x)}%
\end{equation}
and
\begin{equation}
\mathbf{K}_{2}(x,x)=\frac{1}{4}\left(  q(x)-hQ(x)-\frac{Q^{2}(x)}{2}\right).
\label{K2(x,x)}%
\end{equation}
\end{proposition}

\begin{proof}
We obtain from \eqref{IntEq for dtK} and \eqref{GoursatKh2} that
\begin{equation*}
    \mathbf{K}_2(x,x)=\frac 14 q(x) - \frac 12\int_0^x q(z)\mathbf{K}(z,z)\,dz = \frac 14 q(x) - \frac 12\int_0^x q(z)\biggl(\frac h2+\frac 12\int_0^z q(s)\,ds\biggr)\,dz
\end{equation*}
which coincides with \eqref{K2(x,x)} observing that $\int_0^x q(z) \int_0^z q(s)\,ds\,dz=Q^2(x)/2$.

Let us notice that
\[
\mathbf{K}_{1}(x,x)+\mathbf{K}_{2}(x,x) = \frac{d\mathbf{K}(x,x)}{dx}=\frac 12 q(x).
\]
Hence \eqref{K1(x,x)} immediately follows from \eqref{K2(x,x)}.
\end{proof}

Similarly we obtain the following proposition.

\begin{proposition}
The following relations for the partial derivatives of the transmutation
kernel are valid on the line $t=-x$,
\begin{equation}\label{K1(x,-x)}
\mathbf{K}_{1}(x,-x)=\mathbf{K}_{2}(x,-x)=\frac{q(0)}{4} + \frac{h}{4}Q(x).
\end{equation}
\end{proposition}

Relations \eqref{K1(x,x)} and \eqref{K1(x,-x)} were already utilized in \cite{KNT} to estimate the error of the approximation of the derivatives of solutions of equation \eqref{SLhom}.

As can be seen from \eqref{Tf}, the definition of the transmutation operator requires the knowledge of its integral kernel $\mathbf{K}$ in the region $|t|\leq|x|\leq b$. However the integral kernel $\mathbf{K}$ is well-defined (via \eqref{GoursatTh1}--\eqref{GoursatTh2}) in the larger region $\overline{\mathbf{S}}:\ |x|\leq b,\ |t|\leq b$ and is continuously differentiable (twice for $q\in C^1[-b,b]$) there, see \cite{KrT2012}. For the rest of this paper we consider the integral kernel $\mathbf{K}$ to be defined in this larger domain $\overline{\mathbf{S}}$, it allows us to present many results in a simpler and more natural form. First of such results is the following proposition defining the inverse operator $\mathbf{T}_f^{-1}$.

\begin{proposition}[\cite{KrT2012}] \label{Prop Inverse}The inverse operator $\mathbf{T}_{f}^{-1}$
can be represented as the Volterra integral operator
\begin{equation*}
\mathbf{T}_{f}^{-1}u(x)=u(x)-\int_{-x}^{x}\mathbf{K}(t,x)u(t)\,dt.
%\label{Tinverse}%
\end{equation*}
\end{proposition}

Together with the transmutation $\mathbf{T}_{f}$ it is often convenient to
consider another couple of operators enjoying the transmutation property
(\ref{transmutationT}) on subclasses of $C^{2}[-b,b]$ (as well as on
subclasses of $C^{2}[0,b]$), for details see \cite{Marchenko} and additionally
\cite{KT Obzor},
\[
T_{c}w(x)=w(x)+\int_{0}^{x}\mathbf{C}(x,t)w(t)dt
\]
and%
\[
T_{s}w(x)=w(x)+\int_{0}^{x}\mathbf{S}(x,t)w(t)dt
\]
with the kernels $\mathbf{C}$ and $\mathbf{S}$ related to the kernel
$\mathbf{K}$ by the equalities
\[
\mathbf{C}(x,t)=\mathbf{K}(x,t)+\mathbf{K}(x,-t)
\]
and
\[
\mathbf{S}(x,t)=\mathbf{K}(x,t)-\mathbf{K}(x,-t).
\]
The following statement is valid.

\begin{theorem}
[\cite{Marchenko}]\label{TcTsMapsSolutions copy(1)} Solutions $c(\omega,x)$
and $s(\omega,x)$ of equation \eqref{SLomega2} satisfying the initial
conditions
\begin{gather}
c(\omega,0)=1,\qquad c_{x}^{\prime}(\omega,0)=h\label{ICcos}\\
s(\omega,0)=0,\qquad s_{x}^{\prime}(\omega,0)=1 \label{ICsin}%
\end{gather}
can be represented in the form
\begin{equation*}
c(\omega,x)=\cos\omega x+\int_{0}^{x}\mathbf{C}(x,t)\cos\omega t\,dt
\end{equation*}
and
\begin{equation*}
s(\omega,x)=\frac{\sin\omega x}{\omega}+\int_{0}^{x}\mathbf{S}(x,t)\frac
{\sin\omega t}{\omega}\,dt.
\end{equation*}
\end{theorem}

The following important mapping property is a corollary of Theorem
\ref{ThGenSolSturmLiouville}.

\begin{theorem}
[\cite{CKT}, \cite{KrT2012}]\label{Th Transmute}Let $q$ be a continuous
complex valued function of an independent real variable $x\in\lbrack-b,b]$ and
$f$ be a particular solution of \eqref{SLhom} such that $f\neq0$ on $[-b,b]$
and normalized as $f(0)=1$. Denote $h:=f^{\prime}(0)\in\mathbb{C}$. Then
\begin{equation}
\mathbf{T}_{f}\left[  x^{k}\right]  =\varphi_{k}(x)\qquad\text{for any}%
\ k\in\mathbb{N}_{0}. \label{mapping powers 1}%
\end{equation}

\end{theorem}

\begin{remark}
The mapping property \eqref{mapping powers 1} of the transmutation operator
reveals that the SPPS representations \eqref{u1u2} from Theorem
\ref{ThGenSolSturmLiouville} are nothing but the images of Taylor expansions
of the functions $\cosh\sqrt{\lambda}x$ and $\frac{1}{\sqrt{\lambda}}%
\sinh\sqrt{\lambda}x$ under the action of $\mathbf{T}_{f}$.
\end{remark}

In what follows we assume that $f\neq0$ on $[-b,b]$, $f(0)=1$ and denote
$h:=f^{\prime}(0)\in\mathbb{C}$.

We introduce the following two systems of functions
\begin{align}
\mathbf{c}_{m}(x)&=\sum_{\text{even }k=0}^{m}\binom{m}{k}x^{k}\varphi
_{m-k}(x),\quad m=1,2,\ldots\quad\text{ and }\mathbf{c}_{0}(x)=u_{0}(x,x)=f(x),
\label{cm}\\
\mathbf{s}_{m}(x)&=\sum_{\text{odd }k=1}^{m}\binom{m}{k}x^{k}\varphi
_{m-k}(x),\quad m=1,2,\ldots\quad\text{ and }\mathbf{s}_{0}\equiv0. \label{sm}
\end{align}

\begin{example}
In a special case when $f\equiv1$ we obtain that
$\mathbf{c}_{m}(x)=\mathbf{s}_{m}(x)=2^{m-1}x^{m}$, $m=1,2,\ldots$.
\end{example}

The systems of functions $\left\{  \mathbf{c}_{m}\right\}  _{m=0}^{\infty}$
and $\left\{  \mathbf{s}_{m}\right\}  _{m=0}^{\infty}$ arose in \cite{KT
Transmut} as traces of so-called generalized wave polynomials which form a
complete system of solutions of (\ref{GoursatKh1}) and are used in \cite{KT
Transmut} for uniform approximation of the transmutation kernels $\mathbf{K}$,
$\mathbf{C}$ and $\mathbf{S}$.

\begin{definition}[\cite{KKTT}] The following functions
\begin{equation*}
u_{0}=\varphi_{0}(x),\quad u_{2m-1}(x,t)=\sum_{\text{even }k=0}^{m}\binom
{m}{k}\varphi_{m-k}(x)t^{k},\quad u_{2m}(x,t)=\sum_{\text{odd }k=1}^{m}%
\binom{m}{k}\varphi_{m-k}(x)t^{k},%\label{um}%
\end{equation*}
are called generalized wave polynomials (the wave polynomials are introduced
below, in Example \ref{Ex Wave polynomials}). The following parity relations
hold for the generalized wave polynomials.
\begin{equation*}
u_{0}(x,-t)=u_{0}(x,t),\qquad u_{2n-1}(x,-t)=u_{2n-1}(x,t),\qquad
u_{2n}(x,-t)=-u_{2n}(x,t).%\label{umParity}%
\end{equation*}
\end{definition}

\begin{example}
\label{Ex Wave polynomials}In a special case when $f\equiv1$ we obtain
that $\varphi_{k}(x)=x^{k}$, $k\in\mathbb{N}_{0}$ and $u_{k}(x,t)=p_{k}(x,t)$
where $p_{k}$ are wave polynomials \cite[Proposition 1]{KKTT} defined by the
equalities
\begin{align*}
p_{0}(x,t)  & =1,\qquad p_{2m-1}(x,t)=\sum_{\mathrm{even}\text{ }k=0}^{m}%
\binom{m}{k}x^{m-k}t^{k}=\frac{1}{2}\bigl(  \left(  x+t\right)  ^{m}+\left(
x-t\right)  ^{m}\bigr)  ,\\
p_{2m}(x,t)  & =\sum_{\mathrm{odd}\text{ }k=1}^{m}\binom{m}{k}x^{m-k}%
t^{k}=\frac{1}{2}\bigl(  \left(  x+t\right)  ^{m}-\left(  x-t\right)
^{m}\bigr)  .
\end{align*}
\end{example}

\begin{remark}\label{Rem Tf pn}
From Theorem \ref{Th Transmute} one obtains that $u_{n}=\mathbf{T}_{f}\left[
p_{n}\right]  $.
\end{remark}

Recall that $\overline{\mathbf{S}}$ denotes a closed square on the plane $(x,t)$ with a
diagonal joining the endpoints $(b,b)$ and $(-b,-b)$.

\begin{theorem}[\cite{KT AnalyticApprox}]\label{Th Kapprox} Let the complex numbers
$a_{0},\ldots,a_{N}$ and $b_{1},\ldots,b_{N}$ be such that%
\begin{equation}\label{KxxErr}
\left\vert \frac{h}{2}+\frac{1}{4}\int_{0}^{x}q(s)ds-\sum_{n=0}^{N}%
a_{n}\mathbf{c}_{n}(x)\right\vert <\varepsilon_{1}%
\end{equation}
and
\begin{equation}\label{KxmxErr}
\left\vert \frac{1}{4}\int_{0}^{x}q(s)ds-\sum_{n=1}^{N}b_{n}\mathbf{s}%
_{n}(x)\right\vert <\varepsilon_{2}%
\end{equation}
for every $x\in\lbrack-b,b]$. Then the kernel $\mathbf{K}(x,t)$ is
approximated by the function%
\begin{equation}
K_{N}(x,t)=a_{0}u_{0}(x,t)+\sum_{n=1}^{N}a_{n}u_{2n-1}(x,t)+\sum_{n=1}%
^{N}b_{n}u_{2n}(x,t)\label{K(x,t)}%
\end{equation}
in such a way that for every $(x,t)\in\overline{\mathbf{S}}$ the inequality
holds
\begin{equation*}
\bigl|\mathbf{K}(x,t)-K_{N}(x,t)\bigr|\leq\varepsilon%\label{estim}%
\end{equation*}
where $\varepsilon\geq0$ depends on $\varepsilon_{1}$, $\varepsilon_{2}$ and
$q$.
\end{theorem}

\begin{remark}[\cite{KT Transmut}]\label{Rem K as a series} If the functions $\frac{h}%
{2}+\frac{1}{4}\int_{0}^{x}q(s)ds$ and $\frac{1}{4}\int_{0}^{x}q(s)ds$ admit
uniformly convergent on $[-b,b]$ respective series expansions
\[
\frac{h}{2}+\frac{1}{4}\int_{0}^{x}q(s)ds=\sum_{n=0}^{\infty}a_{n}%
\mathbf{c}_{n}(x)\qquad\text{and}\qquad\frac{1}{4}\int_{0}^{x}q(s)ds=\sum
_{n=1}^{\infty}b_{n}\mathbf{s}_{n}(x),
\]
then the kernel $\mathbf{K}$ admits a uniformly convergent in $\overline
{\mathbf{S}}$ series representation
\begin{equation}
\mathbf{K}(x,t)=a_{0}u_{0}(x,t)+\sum_{n=1}^{\infty}a_{n}u_{2n-1}%
(x,t)+\sum_{n=1}^{\infty}b_{n}u_{2n}(x,t).\label{K as a series}%
\end{equation}
\end{remark}

In \cite{KT AnalyticApprox} the uniform approximation of the kernel from
Theorem \ref{Th Kapprox} was applied to obtain a result on the uniform
approximation (with respect to $x$ and $\omega$) of the solutions
$c(\omega,x)$ and $s(\omega,x)$ of equation \eqref{SLomega2}.

\begin{theorem}
[\cite{KT AnalyticApprox}]\label{ThApproxCS}The solutions $c(\omega,x)$ and
$s(\omega,x)$ of equation \eqref{SLomega2} satisfying \eqref{ICcos} and
\eqref{ICsin} respectively can be approximated by the functions
\begin{equation}
c_{N}(\omega,x)=\cos\omega x+2\sum_{n=0}^{N}a_{n}\sum_{\text{even }k=0}%
^{n}\binom{n}{k}\varphi_{n-k}(x)\int_{0}^{x}t^{k}\cos\omega t\,dt \label{cN}%
\end{equation}
and
\begin{equation}
s_{N}(\omega,x)=\frac{1}{\omega}\left(  \sin\omega x+2\sum_{n=1}^{N}b_{n}%
\sum_{\text{odd }k=1}^{n}\binom{n}{k}\varphi_{n-k}(x)\int_{0}^{x}t^{k}%
\sin\omega t\,dt\right)  \label{sN}%
\end{equation}
where the coefficients $\left\{  a_{n}\right\}  _{n=0}^{N}$ and $\left\{
b_{n}\right\}  _{n=1}^{N}$ are the same as in Theorem \ref{Th Kapprox} and the following estimates hold
\begin{equation*}
\left\vert c(\omega,x)-c_{N}(\omega,x)\right\vert \leq\frac{\varepsilon
\sinh(Cx)}{C} %\label{estc2}%
\end{equation*}
and%
\begin{equation*}
\left\vert s(\omega,x)-s_{N}(\omega,x)\right\vert \leq\frac{\varepsilon
\sinh(Cx)}{\left\vert \omega\right\vert C} %\label{estc4}%
\end{equation*}
for any $\omega\in\mathbb{C}$, $\omega\neq0$ belonging to the strip
$\left\vert \operatorname{Im}\omega\right\vert \leq C$, $C\geq0$, where
$\varepsilon\geq0$ depends on $\varepsilon_{1}$, $\varepsilon_{2}$ and $q$.
\end{theorem}

The existence of the appropriate number $N$ and coefficients $\left\{
a_{n}\right\}  _{n=0}^{N}$ and $\left\{  b_{n}\right\}  _{n=1}^{N}$ required
in (\ref{KxxErr}) and (\ref{KxmxErr}) is established by the following statement, see also Section \ref{Sect ConvEstimates}.

\begin{proposition}[\cite{KT AnalyticApprox}] The systems of functions $\left\{  \mathbf{c}%
_{n}\right\}  _{n=0}^{\infty}$ and $\left\{  \mathbf{s}_{n}\right\}
_{n=1}^{\infty}$ are linearly independent and complete in $C^{1}[-b,b]$ and
$C_{0}^{1}[-b,b]$, respectively, with respect to the maximum norm. Here
$C_{0}^{n}[-b,b]$ denotes a subspace of $C^{n}[-b,b]$ consisting of functions
vanishing in the origin.
\end{proposition}

In the next section we show that there exists another way for construction of
the functions $\left\{  \mathbf{c}_{n}\right\}  _{n=0}^{\infty}$ and $\left\{
\mathbf{s}_{n}\right\}  _{n=1}^{\infty}$ which additionally reveals some of
their properties.

\section{Functions $\left\{  \mathbf{c}_{n}\right\}  $ and $\left\{
\mathbf{s}_{n}\right\}  $ as formal powers}
\label{Sect 3}

We shall use the following result from \cite{KrTNewSPPS}. Consider the
Sturm-Liouville equation of the form
\begin{equation}
(p(x)u^{\prime})^{\prime}+q(x)u=\sum_{k=1}^{N}\lambda^{k}R_{k}\left[
u\right]  ,\qquad x\in(a,b) \label{pencil}%
\end{equation}
where $R_{k}$ are linear differential operators of the first order,
$R_{k}\left[  u\right]  :=r_{k}(x)u+s_{k}(x)u^{\prime}$, $k=1,\ldots N$, and the
complex-valued functions $p$, $q$, $r_{k}$, $s_{k}$ are continuous on the
finite segment $\left[  a,b\right]  $.

Define the formal powers for equation \eqref{pencil} as follows
\begin{align*}
\widetilde{\mathcal{X}}^{\left(  -n\right)  }  &  \equiv\mathcal{X}^{\left(
-n\right)  }\equiv0\qquad\text{ for }n\in\mathbb{N},\\%\label{Xpencil1}\\
\widetilde{\mathcal{X}}^{\left(  0\right)  }  &  \equiv\mathcal{X}^{\left(
0\right)  }\equiv1,\\
\displaybreak[2]\widetilde{\mathcal{X}}^{\left(  n\right)  }(x)  &  =%
\begin{cases}
\displaystyle\int_{x_{0}}^{x}f(s)\sum_{k=1}^{N}R_{k}\left[  f(s)\widetilde
{\mathcal{X}}^{\left(  n-2k+1\right)  }(s)\right]  ds, & n\text{ -
odd,}\smallskip\\
\displaystyle\int_{x_{0}}^{x}\widetilde{\mathcal{X}}^{\left(  n-1\right)
}\left(  s\right)  \dfrac{ds}{f^{2}\left(  s\right)  p\left(  s\right)  }, &
n\text{ - even,}%
\end{cases}\\
%\label{eq: equis tilde}\\
\mathcal{X}^{\left(  n\right)  }(x)  &  =%
\begin{cases}
\displaystyle\int_{x_{0}}^{x}\mathcal{X}^{\left(  n-1\right)  }\left(
s\right)  \dfrac{ds}{f^{2}\left(  s\right)  p\left(  s\right)  }, & n\text{ -
odd,}\smallskip\\
\displaystyle\int_{x_{0}}^{x}f(s)\sum_{k=1}^{N}R_{k}\left[  f(s)\mathcal{X}%
^{\left(  n-2k+1\right)  }(s)\right]  ds, & n\text{ - even}%
\end{cases}
%\label{eq: equis sin tilde}%
\end{align*}
where $f$ is a particular complex-valued solution of equation \eqref{pencil} for $\lambda=0$ and $x_{0}$ is an arbitrary point of the segment $\left[  a,b\right]  $ such
that $p(x_{0})\neq0$.

\begin{theorem}
[SPPS representations for polynomial pencils of operators, \cite{KrTNewSPPS}]
\label{ThSPPS_Pencil} Assume that on a finite interval $[a,b]$, the equation
\[
(p(x)v^{\prime})^{\prime}+q(x)v=0
\]
possesses a particular solution $f$ such that the functions $fR_{k}[f]$,
$k=1,\ldots,N$ and $\frac{1}{f^{2}p}$ are continuous on $\left[  a,b\right]
$. Then the general solution of \eqref{pencil} has the form $u=c_{1}%
u_{1}+c_{2}u_{2}$, where $c_{1}$ and $c_{2}$ are arbitrary complex constants
and
\begin{equation}
u_{1}=f\sum_{n=0}^{\infty}\lambda^{n}\widetilde{\mathcal{X}}^{\left(
2n\right)  }\qquad\text{and}\qquad u_{2}=f\sum_{n=0}^{\infty}\lambda
^{n}\mathcal{X}^{\left(  2n+1\right)  }. \label{solPencil}%
\end{equation}
Both series in \eqref{solPencil} converge uniformly on $\left[  a,b\right]  $.
\end{theorem}

Below it becomes clear that of particular interest is the equation
\begin{equation}
y^{\prime\prime}-q(x)y=2\lambda y^{\prime} \label{SL with der on the right}%
\end{equation}
where $q$ is a continuous complex-valued function on $[0,b]$. Its general
solution can be obtained according to Theorem \ref{ThSPPS_Pencil}.

\begin{corollary}
\label{Cor SPPS for Auxiliary SL} Assume that the function $f$ is a solution
of the equation
\[
v^{\prime\prime}-q(x)v=0
\]
on $(0,b)$ such that $f(0)=1$ and $f(x)\neq0$, $x\in\lbrack0,b]$ (see Remark \ref{RemarkNonVanish} according to the
existence of such a nonvanishing solution).
Then a general solution of \eqref{SL with der on the right} can be written in
the form%
\begin{equation}
y_{1}=f\sum_{n=0}^{\infty}\lambda^{n}\widetilde{Y}^{(2n)}\qquad\text{and}\qquad y_{2}=f\sum_{n=0}^{\infty}\lambda^{n}Y^{(2n+1)}
\label{y1 and y2 from Corollary}%
\end{equation}
with the formal powers $\widetilde{Y}^{(n)}$ and $Y^{(n)}$ defined as follows
\begin{align*}
\widetilde{Y}^{(0)}&\equiv Y^{(0)}\equiv1,\\
\widetilde{Y}^{(n)}(x)&=
\begin{cases}
2\int_{0}^{x}f(s)\left(  f(s)\widetilde{Y}^{\left(  n-1\right)
}(s)\right)  ^{\prime}ds, & n\text{ -- odd},\\
\int_{x_{0}}^{x}\widetilde{Y}^{\left(  n-1\right)  }\left(
s\right)  \dfrac{ds}{f^{2}\left(  s\right)  }, & n\text{ -- even},
\end{cases}\\
Y^{(n)}(x)&=
\begin{cases}
2\int_{0}^{x}f(s)\left(  f(s)Y^{\left(  n-1\right)  }(s)\right)
^{\prime}ds, & n\text{ -- even},\\
\int_{0}^{x}Y^{\left(  n-1\right)  }\left(  s\right)  \dfrac
{ds}{f^{2}\left(  s\right)  }, & n\text{ -- odd}.
\end{cases}
\end{align*}
\end{corollary}

This is a direct corollary of Theorem \ref{ThSPPS_Pencil} where $N=1$,
$R_{1}\left[  y\right]  =2y^{\prime}$, etc.

From now on we assume that $f$ satisfies the conditions of Corollary
\ref{Cor SPPS for Auxiliary SL}.

\begin{remark}
\label{Rem Init cond y1 and y2 from Corollary}It is easy to verify that the
solutions $y_{1}$ and $y_{2}$ from \eqref{y1 and y2 from Corollary} fulfil the
following initial conditions
\begin{align*}
y_{1}(0)&=1,& y_{1}^{\prime}(0)&=f^{\prime}(0),\\
y_{2}(0)&=0,& y_{2}^{\prime}(0)&=1.
\end{align*}
\end{remark}

In the following proposition a relation between the functions $\mathbf{c}_{n}$
and $\mathbf{s}_{n}$ defined in \eqref{cm} and \eqref{sm} with the formal powers from Corollary \ref{Cor SPPS for Auxiliary SL} is established.

\begin{proposition}
\label{Prop c_n s_n as formal powers}
\begin{equation}
\frac{\mathbf{c}_{n}}{n!}=fY^{(2n-1)},\qquad\frac{\mathbf{s}_{n}}%
{n!}=f\left(  \widetilde{Y}^{(2n)}+Y^{(2n-1)}\right),\qquad  \text{ for any
odd }n\in\mathbb{N} \label{cn and sn for odd}%
\end{equation}
and
\begin{equation}
\frac{\mathbf{s}_{n}}{n!}=fY^{(2n-1)},\qquad \frac{\mathbf{c}_{n}}{n!}=f\left(  \widetilde{Y}^{(2n)}+Y^{(2n-1)}\right),\qquad\text{for any even
}n\in\mathbb{N}. \label{cn and sn for even}
\end{equation}
\end{proposition}

\begin{proof}
Observe that $u$ is a solution of the equation
\begin{equation}
u^{\prime\prime}-q(x)u=\lambda^{2}u \label{Schr u}%
\end{equation}
iff $v=e^{\lambda x}u$ is a solution of
\[
v^{\prime\prime}-q(x)v=2\lambda v^{\prime}.
\]
Indeed, for $v=e^{\lambda x}u$ we have $v^{\prime}=\lambda v+e^{\lambda
x}u^{\prime}$ and
\begin{equation*}
v^{\prime\prime}    =\lambda v^{\prime}+\lambda e^{\lambda x}u^{\prime
}+e^{\lambda x}u^{\prime\prime} =\lambda v^{\prime}+\lambda\left(  v^{\prime}-\lambda v\right)  +e^{\lambda
x}\left(  qu+\lambda^{2}u\right) =2\lambda v^{\prime}+qv.
\end{equation*}
Analogously, the function $w=e^{-\lambda x}u$ solves the equation
$w^{\prime\prime}-q(x)w=-2\lambda w^{\prime}$.

Now let us consider the solution $u$ of (\ref{Schr u}) having the form
$u=u_{1}+u_{2}$ with the solutions $u_{1}$ and $u_{2}$ constructed according
to Theorem \ref{ThGenSolSturmLiouville},
\[
u_{1}=\sum_{k=0}^{\infty}\frac{\lambda^{2k}}{(2k)!}\varphi_{2k}\qquad
\text{and}\qquad u_{2}=\sum_{k=0}^{\infty}\frac{\lambda^{2k+1}}{(2k+1)!}%
\varphi_{2k+1}.
\]
We have then
\begin{equation}
u(0)=1\qquad\text{and}\qquad u^{\prime}(0)=f^{\prime}(0)+\lambda.
\label{cond for u}%
\end{equation}

Consequently,
\begin{equation}
v    =e^{\lambda x}u=\sum_{k=0}^{\infty}\frac{\lambda^{k}x^{k}}{k!}\sum
_{j=0}^{\infty}\frac{\lambda^{j}}{j!}\varphi_{j}=\sum_{k=0}^{\infty}%
\lambda^{k}\sum_{j=0}^{k}\frac{x^{j}}{j!}\frac{\varphi_{k-j}}{\left(
k-j\right)  !} =\sum_{k=0}^{\infty}\lambda^{k}\left(  \frac{\mathbf{c}_{k}}{k!}%
+\frac{\mathbf{s}_{k}}{k!}\right)  . \label{v1}%
\end{equation}
In a similar way we obtain that
\begin{equation}
w=e^{-\lambda x}u=\sum_{k=0}^{\infty}\lambda^{k}\left(  \frac{\mathbf{c}_{k}%
}{k!}-\frac{\mathbf{s}_{k}}{k!}\right)  . \label{w1}%
\end{equation}

On the other hand, the functions $v$ and $w$ can be constructed according to
Corollary \ref{Cor SPPS for Auxiliary SL}. For this we observe that
\begin{align*}
v(0)&=u(0)=1,& v^{\prime}(0)&=f^{\prime}(0)+2\lambda,\\
w(0)&=1, &w^{\prime}(0)&=f^{\prime}(0).
\end{align*}
Comparing these values with those from Remark
\ref{Rem Init cond y1 and y2 from Corollary} we obtain
\begin{equation}
v=y_{1}+2\lambda y_{2}=f\left(  \sum_{n=0}^{\infty}\lambda^{n}\widetilde
{Y}^{(2n)}+2\sum_{n=0}^{\infty}\lambda^{n+1}Y^{(2n+1)}\right)  \label{v2}%
\end{equation}
and
\begin{equation}
w=f\sum_{n=0}^{\infty}\left(  -\lambda\right)  ^{n}\widetilde{Y}^{(2n)}.
\label{w2}%
\end{equation}
Comparing (\ref{v1}) with (\ref{v2}) and (\ref{w1}) with (\ref{w2}) we arrive
at the relations%
\[
\sum_{k=0}^{\infty}\lambda^{k}\left(  \frac{\mathbf{c}_{k}}{k!}+\frac
{\mathbf{s}_{k}}{k!}\right)  =f+f\sum_{k=1}^{\infty}\lambda^{k}\left(
\widetilde{Y}^{(2k)}+2Y^{(2k-1)}\right)
\]
and
\[
\sum_{k=0}^{\infty}\lambda^{k}\left(  \frac{\mathbf{c}_{k}}{k!}-\frac
{\mathbf{s}_{k}}{k!}\right)  =f\sum_{k=0}^{\infty}(-1)^{k}\lambda^{k}\widetilde{Y}^{(2k)}.
\]
Adding and substracting these two equalities, due to the uniform convergence
of the power series with respect to $\lambda$, we obtain
(\ref{cn and sn for odd}) and (\ref{cn and sn for even}).
\end{proof}

\section{Generalized derivatives}
\label{Sect 4}
Following the ideas from \cite{KMoT}, \cite{KKTT}, \cite{KT Obzor}, let us
introduce the following couple of operations which will be called the
generalized derivatives,%
\[
\gamma_{1}g(x):=f^{2}(x)\frac{d}{dx}\left(  \frac{g(x)}{f(x)}\right)
\]
and%
\[
\gamma_{2}g(x):=\frac{1}{2}\int_{0}^{x}\frac{g^{\prime}(s)ds}{f(s)}.
\]

The following proposition clarifies their relation to equation
(\ref{SL with der on the right}).
\begin{proposition}\label{Prop gamma2gamma1}
For any $g\in C^{2}[0,b]$,
\[
\gamma_{2}\gamma_{1}g(x)=\frac{1}{2}\int_{0}^{x}\left(  \frac{d^{2}}{ds^{2}%
}-q(s)\right)  g(s)ds.
\]
Moreover, the operator $\gamma_2\gamma_1$ can be extended onto $C^1[0,b]$ by the rule
\begin{equation}\label{gamma2gamma1 C1}
\gamma_{2}\gamma_{1}g(x)=\frac{1}{2}\biggl(g'(x)-g'(0) - \int_0^x q(s)g(s)\,ds\biggr).
\end{equation}
\end{proposition}
Indeed, by definition,
\begin{align*}
\gamma_{2}\gamma_{1}g(x)  &  =\frac{1}{2}\int_{0}^{x}\frac{1}{f(s)}\frac
{d}{ds}\left(  f^{2}(s)\frac{d}{ds}\left(  \frac{g(s)}{f(s)}\right)  \right)
ds  =\frac{1}{2}\int_{0}^{x}\frac{1}{f(s)}\frac{d}{ds}\bigl(  f(s)g^{\prime
}(s)-f^{\prime}(s)g(s)\bigr)  ds\\
&  =\frac{1}{2}\int_{0}^{x}\frac{1}{f(s)}\bigl(  f(s)g^{\prime\prime
}(s)-f^{\prime\prime}(s)g(s)\bigr)  ds  =\frac{1}{2}\int_{0}^{x}\bigl(  g^{\prime\prime}(s)-q(s)g(s)\bigr)  ds \\
& = \frac{1}{2}\biggl(g'(x)-g'(0) - \int_0^x q(s)g(s)\,ds\biggr).
\end{align*}
Thus, for a solution of (\ref{SL with der on the right}) we have
\[
\gamma_{2}\gamma_{1}y(x)=\lambda\left(  y(x)-y(0)\right)  .
\]

Note that
\begin{equation*}%\label{gamma2gamma1f}
\gamma_{1}f\equiv0\qquad\text{and}\qquad\gamma_{2}\gamma_{1}\left[  1\right](x)
=-\frac{1}{2}\int_{0}^{x}q(s)ds = -\frac 12 Q(x).
\end{equation*}

\begin{remark}\label{Rem gamma2gamma1 1}
Below we consider the functions $\left(
\gamma_{2}\gamma_{1}\right)  ^{j}\left[  1\right]  $, $j=0,1,\ldots$. They
appear in \cite[Lemma 1.4.1]{Marchenko} where it was shown that they are well
defined for $j\leq n$ if $q\in W_{2}^{n-1}[0,b]$, and the function
$\left(  \gamma_{2}\gamma_{1}\right)  ^{n}\left[  1\right]  $ possesses a
square integrable derivative. Moreover, it can be easily verified using \eqref{gamma2gamma1 C1} that if $q\in C^{n-1}[0,b]$ then the functions $\left(
\gamma_{2}\gamma_{1}\right)  ^{j}\left[  1\right]  $ are well-defined (and are continuous functions) for $j\leq n+1$.

In what follows when considering the function
$\left(  \gamma_{2}\gamma_{1}\right)  ^{j}\left[  1\right]  $, if not specified explicitly, we suppose that
$q\in W_{2}^{j-1}[0,b]$.
\end{remark}

In the following statement we summarize several properties of the functions
$\mathbf{c}_{n}$ and $\mathbf{s}_{n}$ related to the generalized derivatives.

\begin{proposition}
\label{Prop Properties of c_n and s_n} The following relations are valid
\begingroup
\allowdisplaybreaks
\begin{align}
\gamma_{1}\mathbf{c}_{0}&=0,\qquad\gamma_{1} \mathbf{c}_{1}=1,\qquad
\gamma_{1}\mathbf{s}_{1}=f^{2}, \label{gamma_1 c_0}\\
\gamma_{2}\gamma_{1}\mathbf{c}_{0}&=\gamma_{2}\gamma_{1}\mathbf{c}
_{1}=0,\qquad\gamma_{2}\gamma_{1}\mathbf{s}_{1}=f-1,
\label{gamma2 gamma1 c_0}\\
\gamma_{1}\mathbf{c}_{n}(0)&=\gamma_{1}\mathbf{s}_{n}(0)=0,\qquad n=2,3,\ldots, \label{gamma_1 at zero}\\
\gamma_{2}\gamma_{1}\mathbf{c}_{n}&=n\mathbf{s}_{n-1},\qquad
n=1,2,\ldots, \label{gamma2 gamma1 c_n}\\
\gamma_{2}\gamma_{1}\mathbf{s}_{n}&=n\mathbf{c}_{n-1},\qquad
n=2,3,\ldots, \label{gamma2 gamma1 s_n}\\
\left(  \gamma_{2}\gamma_{1}\right)  ^{j}\frac{\mathbf{c}_{n}}{n!}&=
\begin{cases}
\frac{\mathbf{s}_{n-j}}{\left(  n-j\right)  !} & \text{if }j\text{ is odd and
}j<n,\\
\frac{\mathbf{c}_{n-j}}{\left(  n-j\right)  !} & \text{if }j\text{ is even and
}j<n,\\
\left(  \gamma_{2}\gamma_{1}\right)  ^{j-n}\left[  f-1\right]  & \text{if
}j\geq n>0\text{ and }n\text{ is even,}\\
0 & \text{otherwise,}%
\end{cases}\label{gamma2 gamma1 power c}\\
\left(  \gamma_{2}\gamma_{1}\right)  ^{j}\frac{\mathbf{s}_{n}}{n!}&=
\begin{cases}
\frac{\mathbf{c}_{n-j}}{\left(  n-j\right)  !} & \text{if }j\text{ is odd and
}j<n,\\
\frac{\mathbf{s}_{n-j}}{\left(  n-j\right)  !} & \text{if }j\text{ is even and
}j<n,\\
\left(  \gamma_{2}\gamma_{1}\right)  ^{j-n}\left[  f-1\right]  & \text{if
}j\geq n\text{ and }n\text{ is odd,}\\
0 & \text{otherwise.}
\end{cases}
 \label{gamma2 gamma1 power s}
\end{align}
\endgroup
\end{proposition}

\begin{proof}
Consider
\[
\gamma_{1}\mathbf{s}_{1}=f^{2}\left(  \widetilde{Y}^{(2)}+Y^{(1)}\right)
^{\prime}=\widetilde{Y}^{(1)}+1.
\]
Here we used (\ref{cn and sn for odd}) and the definition of the formal powers
$\widetilde{Y}^{(n)}$ and $Y^{(n)}$. Note that
\[
\widetilde{Y}^{(1)}(x)=2\int_{0}^{x}f(s)f^{\prime}(s)ds=f^{2}(x)-1.
\]
Hence $\gamma_{1}\mathbf{s}_{1}=f^{2}$. The other two equalities in
(\ref{gamma_1 c_0}) can are proved in a similar way.

Equalities (\ref{gamma_1 at zero}) follow from Proposition
\ref{Prop c_n s_n as formal powers} and the definition of $\widetilde{Y}%
^{(n)}$ and $Y^{(n)}$.

Equalities (\ref{gamma2 gamma1 c_0}) follow trivially from (\ref{gamma_1 c_0}).

For $n=2,3,\ldots$ consider
\[
\gamma_{2}\gamma_{1}\frac{\mathbf{c}_{n}}{n!}=\frac{1}{2}\gamma_{2}\gamma
_{1}\left[  f\left(  \widetilde{Y}^{(2n)}+(-1)^{n}\widetilde{Y}^{(2n)}%
+2Y^{(2n-1)}\right)  \right]  .
\]
This is according to Proposition \ref{Prop c_n s_n as formal powers}. Hence
\begingroup
\allowdisplaybreaks
\begin{align*}
\gamma_{2}\gamma_{1}\frac{\mathbf{c}_{n}}{n!}  &  =\frac{1}{2}\gamma
_{2}\left[  f^{2}\left(  \left(  1+(-1)^{n}\right)  \widetilde{Y}%
^{(2n)}+2Y^{(2n-1)}\right)  ^{\prime}\right] \\
&  =\frac{1}{2}\gamma_{2}\left[  \left(  1+(-1)^{n}\right)  \widetilde
{Y}^{(2n-1)}+2Y^{(2n-2)}\right] \\
&  =\frac{1+(-1)^{n}}{2}\int_{0}^{x}\left(  f(s)\widetilde{Y}^{(2n-2)}%
(s)\right)  ^{\prime}ds+\int_{0}^{x}\left(  f(s)Y^{(2n-3)}(s)\right)
^{\prime}ds\\
&  =f(x)\left(  \frac{1+(-1)^{n}}{2}\widetilde{Y}^{(2n-2)}(x)+Y^{(2n-3)}%
(x)\right) \\
&  =\frac{1}{2}f(x)\left(  \left(  1-(-1)^{n-1}\right)  \widetilde
{Y}^{(2(n-1))}(x)+2Y^{(2\left(  n-1\right)  -1)}(x)\right) \\
&  =\frac{\mathbf{s}_{n-1}}{\left(  n-1\right)  !}.
\end{align*}
\endgroup
Equality (\ref{gamma2 gamma1 s_n}) can be proved in a similar way.

The proof of (\ref{gamma2 gamma1 power c}) and (\ref{gamma2 gamma1 power s})
follows from the preceding equalities.
\end{proof}

\begin{definition}
Functions of the form
\[
\mathbf{C}_{N}=%
%TCIMACRO{\dsum \limits_{n=0}^{N}}%
%BeginExpansion
{\displaystyle\sum\limits_{n=0}^{N}}
%EndExpansion
\alpha_{n}\frac{\mathbf{c}_{n}}{n!},
\]
where $\alpha_{n}$ are complex numbers, will be called $\mathbf{c}%
$-polynomials of order $N$, and functions of the form
\[
\mathbf{S}_{N}=%
%TCIMACRO{\dsum \limits_{n=1}^{N}}%
%BeginExpansion
{\displaystyle\sum\limits_{n=1}^{N}}
%EndExpansion
\beta_{n}\frac{\mathbf{s}_{n}}{n!},
\]
where $\beta_{n}$ are complex numbers, will be called $\mathbf{s}$-polynomials
of order $N$.
\end{definition}

We stress that in general $\mathbf{C}_{N}$ and $\mathbf{S}_{N}$ are not, of
course, polynomials.

\begin{proposition}
Let $\mathbf{C}_{N}$ be a $\mathbf{c}$-polynomial of order $N$ then its
coefficients are uniquely determined by the relations $\alpha_{0}%
=\mathbf{C}_{N}(0)$ and
\begin{equation}
\alpha_{j+1}=\gamma_{1}\left(  \gamma_{2}\gamma_{1}\right)  ^{j}\mathbf{C}%
_{N}(0)+\sum_{\text{even }n=2}^{j}\alpha_{n}\gamma_{1}\left(  \gamma_{2}%
\gamma_{1}\right)  ^{j-n}\left[  1\right]  (0),\qquad j=0,1,2,\ldots,N-1.
\label{alpha j+1}%
\end{equation}
\end{proposition}

\begin{proof}
Assume $j<N$ being odd. Then due to (\ref{gamma2 gamma1 power c}),
\begin{align*}
\left(  \gamma_{2}\gamma_{1}\right)  ^{j}\mathbf{C}_{N}(x)  &  =\sum_{n=0}^{N}
\alpha_{n}\left(  \gamma_{2}\gamma_{1}\right)  ^{j}\frac{\mathbf{c}_{n}
(x)}{n!}\\
&  =\sum_{\text{even }n=2}^{j}\alpha_{n}\left(  \gamma_{2}\gamma_{1}\right)
^{j-n}\left[  f-1\right]  (x)+\sum_{n=j+1}^{N}\alpha_{n}\frac{\mathbf{s}%
_{n-j}(x)}{\left(  n-j\right)  !}.
\end{align*}
Now applying to this equality $\gamma_{1}$ and considering it at $x=0$ we
obtain
\begin{align*}
\gamma_{1}\left(  \gamma_{2}\gamma_{1}\right)  ^{j}\mathbf{C}_{N}(0)  &
=\sum_{\text{even }n=2}^{j}\alpha_{n}\gamma_{1}\left(  \gamma_{2}\gamma
_{1}\right)  ^{j-n}\left[  f-1\right]  (0)+\sum_{n=j+1}^{N}\alpha_{n}%
\frac{\gamma_{1}\mathbf{s}_{n-j}(0)}{\left(  n-j\right)  !}\\
&  =\sum_{\text{even }n=2}^{j}\alpha_{n}\gamma_{1}\left(  \gamma_{2}\gamma
_{1}\right)  ^{j-n}\left[  f-1\right]  (0)+\alpha_{j+1}.
\end{align*}
Finally, observing that $\gamma_{1}f\equiv0$ we obtain (\ref{alpha j+1}). The
proof of (\ref{alpha j+1}) for any even $j$ is similar. The recursive relation
(\ref{alpha j+1}) implies the uniqueness of the coefficients $\alpha_{n}$.
\end{proof}

Analogously the statement concerning $\mathbf{s}$-polynomials is proved.

\begin{proposition}
Let $\mathbf{S}_{N}$ be an $\mathbf{s}$-polynomial of order $N$ then its
coefficients are uniquely determined by the relations%
\[
\beta_{j+1}=\gamma_{1}\left(  \gamma_{2}\gamma_{1}\right)  ^{j}\mathbf{S}%
_{N}(0)+\sum_{\text{odd }n=1}^{j}\beta_{n}\gamma_{1}\left(  \gamma_{2}%
\gamma_{1}\right)  ^{j-n}\left[  1\right]  (0),\qquad j=0,1,2,\ldots
,N-1.
\]
\end{proposition}

\begin{lemma}
\label{Lemma R(0)}Let the complex valued function $F$ of the variable $x$ be
such that in the neighborhood of $x=0$ its generalized derivatives $\gamma
_{1}\left(  \gamma_{2}\gamma_{1}\right)  ^{j}F$ exist for $j=0,\ldots N-1$ and
are continuous at $x=0$. Let $\mathbf{C}_{N}$ be a $\mathbf{c}$-polynomial
with the coefficients of the form $\alpha_{0}=F(0)$ and
\begin{equation}
\alpha_{j+1}=\gamma_{1}\left(  \gamma_{2}\gamma_{1}\right)  ^{j}%
F(0)+\sum_{\text{even }n=2}^{j}\alpha_{n}\gamma_{1}\left(  \gamma_{2}%
\gamma_{1}\right)  ^{j-n}\left[  1\right]  (0),\qquad j=0,1,2,\ldots.
\label{alpha j+1 F}%
\end{equation}
Then for the function $R:=F-\mathbf{C}_{N}$ the following relations are valid%
\begin{equation}
R(0)=\gamma_{1}R(0)=\ldots=\gamma_{1}\left(  \gamma_{2}\gamma_{1}\right)
^{N-1}R(0)=0. \label{Rat0}%
\end{equation}

\end{lemma}

\begin{proof}
Obviously, $R(0)=F(0)-\mathbf{C}_{N}(0)=F(0)-\alpha_{0}=0$ and $\gamma
_{1}R(0)=\gamma_{1}F(0)-\gamma_{1}\mathbf{C}_{N}(0)=\gamma_{1}F(0)-\alpha
_{1}=0$. For any $j=1,\ldots N-1$ we obtain%
\begin{align*}
\gamma_{1}\left(  \gamma_{2}\gamma_{1}\right)  ^{j}R(0)  &  =\gamma_{1}\left(
\gamma_{2}\gamma_{1}\right)  ^{j}F(0)-\gamma_{1}\left(  \gamma_{2}\gamma
_{1}\right)  ^{j}\mathbf{C}_{N}(0)\\
&  =\gamma_{1}\left(  \gamma_{2}\gamma_{1}\right)  ^{j}F(0)-\alpha_{j+1}%
+\sum_{\text{even }n=2}^{j}\alpha_{n}\gamma_{1}\left(  \gamma_{2}\gamma
_{1}\right)  ^{j-n}\left[  1\right]  (0)=0
\end{align*}
due to (\ref{alpha j+1 F}).
\end{proof}

A similar statement is true for $\mathbf{s}$-polynomials.

\begin{lemma}
\label{Lemma R(0) beta}Let the complex valued function $F$ of the variable $x$
be such that in the neighborhood of $x=0$ there exist $\gamma_{1}\left(
\gamma_{2}\gamma_{1}\right)  ^{j}F$ for $j=0,\ldots N-1$ and are continuous at
$x=0$, $F(0)=0$. Let $\mathbf{S}_{N}$ be an $\mathbf{s}$-polynomial with the
coefficients of the form
\begin{equation}
\beta_{j+1}=\gamma_{1}\left(  \gamma_{2}\gamma_{1}\right)  ^{j}F(0)+\sum
_{\text{odd }n=1}^{j}\beta_{n}\gamma_{1}\left(  \gamma_{2}\gamma_{1}\right)
^{j-n}\left[  1\right]  (0),\qquad j=0,1,2,\ldots,N-1. \label{beta j+1 F}%
\end{equation}
Then for the function $R:=F-\mathbf{S}_{N}$ the following relations are valid%
\[
R(0)=\gamma_{1}R(0)=\ldots=\gamma_{1}\left(  \gamma_{2}\gamma_{1}\right)
^{N-1}R(0)=0.
\]

\end{lemma}

\begin{lemma}
\label{Lemma o(x)} Suppose the complex valued function $R$ of the variable $x$
be such that in the neighborhood of $x=0$ there exist $\gamma_{1}\left(
\gamma_{2}\gamma_{1}\right)  ^{j}R$ for $j=0,\ldots N-1$ and are continuous at
$x=0$. Let
\[
R(0)=\gamma_{1}R(0)=\ldots=\gamma_{1}\left(  \gamma_{2}\gamma_{1}\right)
^{n}R(0)=0.
\]
Then $R(x)=o(x^{n+1}),\quad x\rightarrow0$.
\end{lemma}

\begin{proof}
The proof is by induction. For $n=0$ by assumption we have that $R(0)=\gamma
_{1}R(0)=0$. Note that by definition of $\gamma_{1}$, $\gamma_{1}%
R(0)=-f^{\prime}(0)R(0)+R^{\prime}(0)$ (we took into account that $f(0)=1$),
and hence from the assumption we obtain that $R^{\prime}(0)=0$ which implies
that $R(x)=o(x)$.

Now, assuming that the statement is true for $n-1$, let us prove it for $n$.
For this, consider $R_{1}:=\gamma_{2}\gamma_{1}R$. Then by assumption we have
$\gamma_{1}R_{1}(0)=\ldots=\gamma_{1}\left(  \gamma_{2}\gamma_{1}\right)
^{n-1}R_{1}(0)=0$ as well as $R_{1}(0)=0$ due to the definition of $\gamma
_{2}$. Hence $R_{1}(x)=o(x^{n})$. Also we observe that
\begin{equation}
R_{1}(x)=\frac{1}{2}\int_{0}^{x}\left(  R^{\prime\prime}(s)-q(s)R(s)\right)
ds=\frac{1}{2}\left(  R^{\prime}(x)-\int_{0}^{x}q(s)R(s)ds\right)  ,
\label{R1}%
\end{equation}
where we took into account that $R^{\prime}(0)=0$.

Due to the mean-value theorem we have%
\[
R(x)=R(x)-R(0)=x\left(  \operatorname*{Re}R^{\prime}(c_{1}%
)+i\operatorname*{Im}R^{\prime}(c_{2})\right)  ,
\]
where $c_{1}$ and $c_{2}$ are some points between $0$ and $x$. Using
(\ref{R1}) we obtain
\begin{align*}
R(x)  &  =x\operatorname*{Re}\left(  2R_{1}(c_{1})+\int_{0}^{c_{1}%
}q(s)R(s)ds\right)  +ix\operatorname*{Im}\left(  2R_{1}(c_{2})+\int_{0}%
^{c_{2}}q(s)R(s)ds\right) \\
&  =x\left(  \operatorname*{Re}\left(  o(x^{n}\right)  +o(x^{n+1}%
))+i\operatorname{Im}\left(  o(x^{n}\right)  +o(x^{n+1}))\right)  =o(x^{n+1}).
\end{align*}
\end{proof}

\begin{theorem}[Taylor-type theorem with the Peano form of the remainder term] \label{Thm Taylor} Let the
complex valued function $F$ of the variable $x$ be such that in the
neighborhood of $x=0$ its generalized derivatives $\gamma_{1}\left(
\gamma_{2}\gamma_{1}\right)  ^{j}F$ exist for $j=0,\ldots N-1$ and are
continuous at $x=0$. Then
\begin{equation}
F(x)=\sum_{n=0}^{N}
\alpha_{n}\frac{\mathbf{c}_{n}(x)}{n!}+o(x^{N}) \label{Fc}
\end{equation}
where the coefficients $\alpha_{n}$ are defined as in Lemma \ref{Lemma R(0)}.

If additionally $F(0)=0$, then also
\begin{equation}
F(x)=\sum_{n=1}^{N}
\beta_{n}\frac{\mathbf{s}_{n}(x)}{n!}+o(x^{N}) \label{Fs}
\end{equation}
where the coefficients $\beta_{n}$ are defined as in Lemma
\ref{Lemma R(0) beta}.
\end{theorem}

\begin{proof}
For the proof of (\ref{Fc}) one needs to observe that due to Lemma
\ref{Lemma R(0)}, $R:=F-\sum_{n=0}^{N}
\alpha_{n}\frac{\mathbf{c}_{n}}{n!}$ satisfies (\ref{Rat0}) and hence
according to Lemma \ref{Lemma o(x)}, $R(x)=o(x^{N})$, $x\rightarrow0$.

Equality (\ref{Fs}) is proved similarly with the aid of Lemmas
\ref{Lemma R(0) beta} and \ref{Lemma o(x)}.
\end{proof}

\begin{remark}
Consider the solution $u$ of \eqref{Schr u} introduced in the proof of
Proposition \ref{Prop c_n s_n as formal powers} as satisfying
\eqref{cond for u}. Then the function $\cosh\lambda x\cdot u(x)$ is an example
of a function admitting a uniformly convergent on $[0,b]$ Taylor-type
expansion in terms of the functions $\mathbf{c}_{n}$. Indeed, we have
\[
\cosh\lambda x\cdot u(x)=\sum_{n=0}^{\infty}\lambda^{n}\frac{\mathbf{c}%
_{n}(x)}{n!}%
\]
and analogously,%
\[
\sinh\lambda x\cdot u(x)=\sum_{n=1}^{\infty}\lambda^{n}\frac{\mathbf{s}%
_{n}(x)}{n!}.
\]
The relations are obtained by adding and subtracting the functions $v$ and
$w$ from the proof of Proposition \ref{Prop c_n s_n as formal powers}.
\end{remark}

\begin{example}\label{Ex a2n b2n}
Lemmas \ref{Lemma R(0)} and \ref{Lemma R(0) beta} easily give us the exact expressions for several first coefficients $\alpha_n$ and $\beta_n$ of the generalized Taylor formulas for the functions $F_1(x) = \frac h2 + \frac 14\int_0^x q(s)\,ds$ and $F_2(x) = \frac 14\int_0^x q(s)\,ds$ appearing in \eqref{KxxErr}, \eqref{KxmxErr} and in the representation \eqref{K as a series}, see also \eqref{Goursat conditions Taylor}. First, observe that for a function $g$ such that $g(0)=0$ one has $\gamma_1 g(0) = g'(0)$. Second, we can compute several expressions $(2\gamma_2\gamma_1)^j [1]$ using \eqref{gamma2gamma1 C1}. We have
\begin{align*}
    (2\gamma_2\gamma_1)[1](x) &= -Q(x),\\
    (2\gamma_2\gamma_1)^2[1](x) &= -q(x) + q(0) - \frac 12 Q^2(x),\\
    (2\gamma_2\gamma_1)^3[1](x) &= -q'(x) + q'(0) - (q(x) + q(0))Q(x) + \int_0^x q^2(s)\,ds + \frac 16 Q^3(x),\\
    (2\gamma_2\gamma_1)^4[1](x) &= -q''(x) + q''(0) -q'(x)Q(x) - q(0)q(x) \\
    & \quad + q^2(0) + \frac 12 q(x)Q^2(x) - \int_0^x q(s) (2\gamma_2\gamma_1)^3[1](s)\,ds.
\end{align*}
Note that by the definition of the operator $\gamma_2\gamma_1$ all the expressions above are equal to zero at $x=0$. Hence,
\begin{gather*}
    \gamma_1[1](0)  = -h, \qquad   \gamma_1(2\gamma_2\gamma_1)[1](0)  = -q(0), \qquad \gamma_1(2\gamma_2\gamma_1)^2[1](0) = -q'(0), \\ \gamma_1(2\gamma_2\gamma_1)^3[1](0) = -q''(0), \qquad
    \gamma_1(2\gamma_2\gamma_1)^4[1](0)  = -q'''(0) - 2q'(0)q(0).
\end{gather*}
For Lemma \ref{Lemma R(0)}, the function $F_1(x) = \frac h2 - \frac 12 (\gamma_2\gamma_1)[1]$, hence we obtain
\begin{gather*}
    \alpha_0 = \frac h2, \qquad \alpha_1 = \frac{q(0)}4 - \frac{h^2}2,\qquad \alpha_2 = \frac{q'(0)}8 - \frac{h q(0)}4,\\
    \alpha_3 = \frac{q''(0)}{16} - \frac{hq'(0)}4 + \frac{h^2 q(0)}4,\qquad \alpha_4=\frac{q'''(0)}{32} - \frac{h\bigl(q''(0)-2q^2(0)\bigr)}{16}.
\end{gather*}
For Lemma \ref{Lemma R(0) beta}, the function $F_2(x) = - \frac 12 (\gamma_2\gamma_1)[1]$, hence we obtain
\begin{gather*}
    \beta_1 = \frac{q(0)}4,\qquad \beta_2=\frac{q'(0)}8 - \frac{h q(0)}4,\\
    \beta_3 = \frac{q''(0)}{16} - \frac{q^2(0)}8,\qquad \beta_4=\frac{q'''(0)}{32} - \frac{h\bigl(q''(0)-2q^2(0)\bigr)}{16}.
\end{gather*}
Observe that the coefficients with the even indices coincide. Later in Propositions \ref{Prop a2n b2n} and \ref{Prop a2n b2n 2} we give the proof of this phenomenon.
\end{example}

\section{The preimage of the transmutation kernel} % and some useful relations}
\label{Sect 5}

Together with the function $\mathbf{K}$ let us consider its preimage
$\mathbf{k}:=\mathbf{T}_{f}^{-1}\left[  \mathbf{K}\right]  $. Due to Theorem
\ref{Th Transmutation} it is a solution of the wave equation and hence admits
the representation%
\begin{equation}
\mathbf{k}(x,\tau)=\varphi\left(  \frac{x+\tau}{2}\right)  +\psi\left(
\frac{x-\tau}{2}\right)  , \label{k via phi psi}%
\end{equation}
where the functions $\varphi$ and $\psi$ are unique up to an additive constant.

\begin{proposition}
The function $\mathbf{k}(x,\tau)$ admits the representation%
\begin{equation}\label{k(x,tau)}
\begin{split}
\mathbf{k}(x,\tau)  &  =\frac{h}{2}+\frac{1}{2}\int_{0}^{\frac{x+\tau}{2}}
q(s)ds-\int_{-\frac{x+\tau}{2}}^{\frac{x+\tau}{2}}
\mathbf{K}^{2}\left(t,\frac{x+\tau}{2}\right)dt\\
&  \quad-\int_{-\frac{x-\tau}{2}}^{\frac{x-\tau}{2}}
\mathbf{K}\left(t,\frac{x-\tau}{2}\right)\mathbf{K}\left(t,-\frac{x-\tau}{2}\right)\,dt.
\end{split}
\end{equation}
\end{proposition}

\begin{proof}
From Proposition \ref{Prop Inverse} we have
\begin{equation*}
\mathbf{k}(x,\tau)=\mathbf{K}(x,\tau)-\int_{-x}^{x}
\mathbf{K}(t,x)\mathbf{K}(t,\tau)dt. %\label{k as T_-1}%
\end{equation*}
From this equality and (\ref{k via phi psi}) we obtain%
\begin{equation*}
\mathbf{k}(x,x)    =\varphi\left(  x\right)  +\psi(0)  =\mathbf{K}(x,x)-\int_{-x}^{x}
\mathbf{K}^{2}(t,x)dt  =\frac{h}{2}+\frac{1}{2}\int_{0}^{x}
q(s)ds-\int_{-x}^{x}
\mathbf{K}^{2}(t,x)dt
\end{equation*}
where (\ref{GoursatKh2}) was used.

Similarly,
\begin{equation}
\mathbf{k}(x,-x)   =\varphi\left(  0\right)  +\psi(x)
  =\mathbf{K}(x,-x)-\int_{-x}^{x}
\mathbf{K}(t,x)\mathbf{K}(t,-x)dt
 =\frac{h}{2}-\int_{-x}^{x}
\mathbf{K}(t,x)\mathbf{K}(t,-x)dt. \label{psi (x)}%
\end{equation}
Note that
\begin{equation}
\mathbf{k}(0,0)=\varphi\left(  0\right)  +\psi(0)=\frac{h}{2}. \label{k(0,0)}%
\end{equation}
Thus,
\[
\varphi\left(  \frac{x+\tau}{2}\right)  =\frac{h}{2}+\frac{1}{2}\int_{0}^{\frac{x+\tau}{2}}
q(s)ds-\int_{-\frac{x+\tau}{2}}^{\frac{x+\tau}{2}}
\mathbf{K}^{2}\left(t,\frac{x+\tau}{2}\right)dt-\psi(0)
\]
and
\[
\psi\left(  \frac{x-\tau}{2}\right)  =\frac{h}{2}-\int_{-\frac{x-\tau}{2}}^{\frac{x-\tau}{2}}
\mathbf{K}\left(t,\frac{x-\tau}{2}\right)\mathbf{K}\left(t,-\frac{x-\tau}{2}\right)dt-\varphi\left(
0\right)  .
\]
Adding these expressions and taking into account (\ref{k(0,0)}) we obtain
(\ref{k(x,tau)}).
\end{proof}

\begin{remark}
\label{Rem psi prime is even}The function $\psi^{\prime}$ is even. Indeed, due
to \eqref{psi (x)},
\begin{equation*}
\begin{split}
\psi(-x)    &=\frac{h}{2}-\varphi\left(  0\right)  -\int_{x}^{-x}
\mathbf{K}(t,-x)\mathbf{K}(t,x)dt
  \\& =\frac{h}{2}-\varphi\left(  0\right)  +\int_{-x}^{x}
\mathbf{K}(t,x)\mathbf{K}(t,-x)dt
  =-\psi(x)+h-2\varphi\left(  0\right)  .
\end{split}
\end{equation*}

From this fact and from \eqref{k via phi psi} it follows that
\[
\mathbf{k}_{1}(x,\tau)=\mathbf{k}_{1}(\tau,x)\qquad\text{and}\qquad
\mathbf{k}_{2}(x,\tau)=\mathbf{k}_{2}(\tau,x).
\]
\end{remark}

\begin{proposition}\label{Prop a2n b2n}
Suppose that the conditions from Remark \ref{Rem K as a series} are fulfilled
and hence $\mathbf{K}$ admits the representation \eqref{K as a series}. Then
\begin{equation}
a_{2n}=b_{2n}\qquad\text{for any }n\in\mathbb{N}.\label{a2n=b2n}%
\end{equation}
\end{proposition}
\begin{proof}
Indeed, in this case
\begin{align*}
\mathbf{k}(x,\tau)  & =a_{0}p_{0}(x,\tau)+\sum_{n=1}^{\infty}\left(
a_{n}p_{2n-1}(x,\tau)+b_{n}p_{2n}(x,\tau)\right)  \\
& =\frac{h}{2}+\frac{1}{2}\sum_{n=1}^{\infty}\left(  \left(  a_{n}%
+b_{n}\right)  (x+\tau)^{n}+\left(  a_{n}-b_{n}\right)  (x-\tau)^{n}\right)  .
\end{align*}
Thus,
\[
\varphi\left(  \frac{x+\tau}{2}\right)  =\varphi\left(  0\right)  +\frac{1}%
{2}\sum_{n=1}^{\infty}\left(  a_{n}+b_{n}\right)  (x+\tau)^{n}%
\]
and
\[
\psi\left(  \frac{x-\tau}{2}\right)  =\psi\left(  0\right)  +\frac{1}{2}%
\sum_{n=1}^{\infty}\left(  a_{n}-b_{n}\right)  (x-\tau)^{n}.
\]
From the last equality we have
\[
\psi\left(  x\right)  =\psi\left(  0\right)  +\frac{1}{2}\sum_{n=1}^{\infty
}2^{n}\left(  a_{n}-b_{n}\right)  x^{n}%
\]
and
\[
\psi\left(  -x\right)  =\psi\left(  0\right)  +\frac{1}{2}\sum_{n=1}^{\infty
}\left(  -1\right)  ^{n}2^{n}\left(  a_{n}-b_{n}\right)  x^{n}.
\]
Due to Remark \ref{Rem psi prime is even} we obtain the equality
\[
\sum_{n=1}^{\infty}\left(  -1\right)  ^{n}2^{n}\left(  a_{n}-b_{n}\right)
x^{n}=-\sum_{n=1}^{\infty}2^{n}\left(  a_{n}-b_{n}\right)  x^{n}%
\]
which proves the equality of the even coefficients \eqref{a2n=b2n}.
\end{proof}

\section{Goursat-to-Goursat transmutation operators and the generalized derivatives}\label{Sect 6}
By $\overline{\mathbf{S}}$ we denote a closed square with a diagonal joining
the endpoints $(b,b)$ and $(-b,-b)$ (c.f., Sect. \ref{Subsection Transmutation Operators}).  Let $\square :=\partial
_{x}^{2}-\partial _{t}^{2}$ and the functions $\widetilde{u}$ and $u$ be
solutions of the equations $\square \widetilde{u}=0$ and $\left( \square
-q(x)\right) u=0$ in $\overline{\mathbf{S}}$, respectively such that $u=T_{f}
\widetilde{u}$. In \cite{KT AnalyticApprox} the operator $G$ acting on the space $C^{1}[-b,b]\times C_{0}^{1}[-b,b]$ as
\begin{equation}\label{opG}
G:\ \frac{1}{2}\begin{pmatrix}
\widetilde{u}(x,x)+\widetilde{u}(x,-x)\\
\widetilde{u}(x,x)-\widetilde{u}(x,-x)
\end{pmatrix} \longmapsto \frac{1}{2}\begin{pmatrix}
u(x,x)+u(x,-x)\\
u(x,x)-u(x,-x)\end{pmatrix}.
\end{equation}
was introduced. It is  bounded together with its inverse and is related to the operator transforming the Goursat data for the equation $\square \widetilde{u}=0$ into the Goursat data for the equation $\left( \square-q(x)\right) u=0$. The operator $G$ allowed us to prove the completeness of the functions $\left\{ \mathbf{c}_{n}\right\} _{n=0}^{\infty }$ in $C^{1}[-b,b]$ and of the functions $\left\{ \mathbf{s}_{n}\right\} _{n=1}^{\infty }$ in $C_0^{1}[-b,b]$, see
\cite[Corollary 4.5]{KT AnalyticApprox}. The following proposition summarizes some properties of the operator $G$.
\begin{proposition}[\cite{KT AnalyticApprox}]\label{Prop Operator G}\mbox{}
\begin{enumerate}
\item[1)] The operator $G$ defined by \eqref{opG} on $C^{1}[-b,b]\times
C_{0}^{1}[-b,b]$ admits the following representation
\begin{equation*}
G\begin{pmatrix}\eta (x)\\
\xi (x)\end{pmatrix}=\begin{pmatrix}G_1[\eta (x)]\\
G_2[\xi (x)]\end{pmatrix}=\begin{pmatrix}G_{+}\left[ \eta (x)-\frac{\eta (0)
}{2}\right] +\frac{\eta (0)}{2}\\
G_{-}\left[ \xi (x)\right] \end{pmatrix}
\end{equation*}
where $G_{+}$ and $G_{-}$ have the form
\begin{equation*}
G_{\pm }\eta (x)=\eta (x)+\int_{-x}^{x}\mathbf{K}(x,t)\left( \eta \left(\frac{t+x}{2}\right)\pm
\eta \left(\frac{t-x}{2}\right)\right) dt.
%\label{Gform2}
\end{equation*}
\item[2)] Both operators $G_{+}$ and $G_{-}$ preserve the value of the function in
the origin and $G_{+}:C^{1}[-b,b]\rightarrow C^{1}[-b,b]$, $%
G_{-}:C_{0}^{1}[-b,b]\rightarrow C_{0}^{1}[-b,b]$.
\item[3)] There exist the inverse operators $G_{+}^{-1}$ and $G_{-}^{-1}$ as well as the inverse operators $G_{1}^{-1}$ and $G_{2}^{-1}$  defined
on $C^{1}[-b,b]$, and the inverse
operator for $G$ admits the representation
\begin{equation*}
G^{-1}\begin{pmatrix}\eta (x)\\
\xi (x)\end{pmatrix}=\begin{pmatrix}G_{+}^{-1}\left[ \eta (x)-\frac{\eta
(0)}{2}\right] +\frac{\eta (0)}{2}\\
G_{-}^{-1}\left[ \xi (x)\right] \end{pmatrix}.
%\label{G-1}
\end{equation*}
\item[4)] The operator $G$ maps the powers of $x$ into the functions $\mathbf{c}_n$ and $\mathbf{s}_n$ according to the following relations
\begin{equation*}
        G:\ \begin{pmatrix}1\\0\end{pmatrix} \mapsto  \begin{pmatrix}\mathbf{c}_{0}(x)\\0\end{pmatrix},\quad
        G:\ 2^{n-1}\begin{pmatrix}x^{n}\\
0\end{pmatrix} \mapsto  \begin{pmatrix}\mathbf{c}_{n}(x)\\
0\end{pmatrix},\quad
G:\ 2^{n-1}\begin{pmatrix}0\\
x^{n}\end{pmatrix} \mapsto  \begin{pmatrix}0\\
\mathbf{s}_{n}(x)\end{pmatrix}.
\end{equation*}
\end{enumerate}
\end{proposition}

Together with the operators $G_+$ and $G_-$ consider the following pair of operators acting on $C^1[-b,b]$
\begin{equation*}%\label{operatorGamma-}
    \Gamma_-\eta(x)=\eta(0) + \int_{-x}^x \mathbf{K}(x,t)\eta\left(\frac{t-x}2\right)\,dt
\end{equation*}
and
\begin{equation*}%\label{operatorGamma+}
    \Gamma_+\eta(x)=\eta(x) + \int_{-x}^x \mathbf{K}(x,t)\eta\left(\frac{t+x}2\right)\,dt.
\end{equation*}

\begin{proposition}\label{Prop Gammapm xn}
The images of the powers of $x$ under the action of the operators $\Gamma_{\pm}$ are given by the following relations
\begin{equation}\label{Gamma_pm xn}
    \Gamma_-[x^n](x) = \frac 1{2^n}\bigl(\mathbf{c}_n(x)-\mathbf{s}_n(x)\bigr),\qquad
    \Gamma_+[x^n](x) = \frac 1{2^n}\bigl(\mathbf{c}_n(x)+\mathbf{s}_n(x)\bigr),\qquad n\ge 0.
\end{equation}
\end{proposition}

\begin{proof}
We present the proof for the first relation in \eqref{Gamma_pm xn}, the proof for the second relation is similar.

If $n=0$, then
\begin{equation*}
    \Gamma_-[1](x)=1+\int_{-x}^x \mathbf{K}(x,t) \cdot 1\,dt=\mathbf{T}_f[1](x)=\mathbf{c}_0(x) = \mathbf{c}_0(x) - \mathbf{s}_0(x).
\end{equation*}

If $n>0$, then using Remark \ref{Rem Tf pn} and formulas from Example \ref{Ex Wave polynomials} we obtain
\begin{equation*}
    \begin{split}
       \Gamma_-[x^n](x) & = \int_{-x}^x \mathbf{K}(x,t) \left(\frac{t-x}2\right)^n\,dt =\frac 1{2^n}\int_{-x}^x \mathbf{K}(x,t) \bigl(p_{2n-1}(t,x) - p_{2n}(t,x)\bigr)\,dt\\
         & = \frac 1{2^n}\bigl(p_{2n-1}(x,x) - p_{2n}(x,x)\bigr) + \frac 1{2^n}\int_{-x}^x \mathbf{K}(x,t) \bigl(p_{2n-1}(t,x) - p_{2n}(t,x)\bigr)\,dt\\
         & = \frac 1{2^n}\bigl(u_{2n-1}(x,x) - u_{2n}(x,x)\bigr)=\frac 1{2^n}\bigl(\mathbf{c}_{n}(x) - \mathbf{s}_{n}(x)\bigr).
     \end{split}
\end{equation*}
\end{proof}

\begin{proposition}
The following commutation relations hold for the operators $\Gamma_-$ and $\Gamma_+$ on $C^2[-b,b]$,
\begin{equation}\label{GammasCommutation}
    \bigl(\partial^2-q(x)\bigr)\Gamma_+ = \partial\Gamma_+\partial\qquad\text{and}\qquad \bigl(\partial^2-q(x)\bigr)\Gamma_- = -\partial\Gamma_-\partial,
\end{equation}
moreover, on $C^1[-b,b]$ the operators $\Gamma_{\pm}$ satisfy
\begin{equation}\label{GammasCommutation2}
    2\gamma_2\gamma_1\Gamma_+ = \bigl(\Gamma_+-\delta)\partial\qquad\text{and}\qquad 2\gamma_2\gamma_1\Gamma_- = -(\Gamma_--\delta)\partial,
\end{equation}
where $\partial$ denotes $d/dx$ and $\delta [\eta](x) = \eta(0)$.
\end{proposition}
\begin{proof}
Note that the relations \eqref{GammasCommutation} and \eqref{GammasCommutation2} are equivalent on $C^2[-b,b]$ due to Proposition \ref{Prop gamma2gamma1}.

Relations \eqref{GammasCommutation} can be verified directly similarly to \cite[Theorem 6]{KrT2012}. We omit this lengthy verification and exploit Proposition \ref{Prop Gammapm xn} to verify \eqref{GammasCommutation2}.

Let $n>1$. Then using Propositions \ref{Prop Properties of c_n and s_n} and \ref{Prop Gammapm xn}
\[
2\gamma_2\gamma_1\Gamma_{\pm}[x^n]=2\gamma_2\gamma_1\left[\frac 1{2^n}\bigl(\mathbf{c}_n\pm\mathbf{s}_n\bigr)\right]=\frac n{2^{n-1}}\bigl(\mathbf{s}_n\pm\mathbf{c}_n\bigr)=\pm \Gamma_{\pm}[nx^{n-1}]=\pm \bigl(\Gamma_{\pm}-\delta\bigr)[\partial x^n],
\]
i.e., \eqref{GammasCommutation2} holds for the functions $x^n$ for any $n>1$. The validity of \eqref{GammasCommutation2} for the cases $n=0,1$ can be easily verified as well using Proposition \ref{Prop Properties of c_n and s_n}. By linearity \eqref{GammasCommutation2} holds for any polynomial.

Consider a function $\eta(x)\in C^1[-b,b]$ and let $p_n$ be a sequence of polynomials such that $p_n(0)=\eta(0)$, $p'_n(0)=\eta'(0)$ and $p'_n(x)\to \eta'(x)$ uniformly as $n\to\infty$. Then necessarily $p_n(x)\to \eta(x)$ uniformly for $x\in[-b,b]$. Using Proposition \ref{Prop gamma2gamma1} the relations \eqref{GammasCommutation2} for the polynomials $p_n$ can be written as
\[
\frac{d}{dx}\bigl(\Gamma_{+}[p_n(x)]\bigr) = p_n'(0) + hp_n(0) + \int_{0}^x q(s)\Gamma_{+}[p_n(s)]\,ds + \Gamma_+[p'_n(x)]-p'_n(0)
\]
and
\[
\frac{d}{dx}\bigl(\Gamma_{-}[p_n(x)]\bigr) = hp_n(0) + \int_{0}^x q(s)\Gamma_{-}[p_n(s)]\,ds - \Gamma_-[p'_n(x)]+p'_n(0)
\]
The right-hand sides of these expressions converge uniformly as $n\to\infty$. The differentiation operators in the left-hand sides are closed operators (defined on suitable subspaces of $C^1[-b,b]$) hence allowing us to pass to the limit as $n\to\infty$ and verify \eqref{GammasCommutation2} for the function $\eta$.
\end{proof}

The operators $G_1$ and $G_2$ have the following representations in terms of the operators $\Gamma_1$ and $\Gamma_2$,
\begin{equation}\label{G1G2Gammas}
    G_1=\Gamma_++\Gamma_--f\cdot \delta,\qquad \text{and}\qquad G_2=\Gamma_+-\Gamma_- +\delta.
\end{equation}
Applying the generalized derivatives $\gamma_2\gamma_1$ to \eqref{G1G2Gammas} and using \eqref{GammasCommutation2} we obtain the following corollary.
\begin{corollary}\label{Cor gamma2gamma1 commut}
The following commutation relations hold for the operators $G_1$ and $G_2$ on $C^1[-b,b]$,
\begin{equation}\label{G1G2commut}
    2\gamma_2\gamma_1 G_1 = \bigl(G_2-\delta\bigr)\partial\qquad \text{and}\qquad
    2\gamma_2\gamma_1 \bigl(G_2-\delta\bigr) = \bigl(G_1+(f-2)\delta\bigr)\partial.
\end{equation}
\end{corollary}

Repeated application of Corollary \ref{Cor gamma2gamma1 commut} allows us to obtain the following result.
\begin{corollary}\label{Cor G1G2 uCn}
Suppose that $q\in C^p[-b,b]$ and let a function $u\in C^m[-b,b]$ be given. Then there exist all the functions $(\gamma_2\gamma_1)^j [G_1u]$ for $j\le \min\{m, p+4\}$ and all the functions $(\gamma_2\gamma_1)^j\linebreak[2] [G_2u-u(0)]$ for $j\le\min\{m, p+3\}$ and they are given by the following expressions.

If $n=2k$ then
\begin{align}
    (2\gamma_2\gamma_1)^{n} G_1 u &=
    G_1 u^{(n)} + f\cdot  u^{(n)}(0)- 2\sum_{\text{even } j=2}^{n}u^{(j)}(0)\cdot(2\gamma_2\gamma_1)^{n-j}[1],\label{CommuteG1 2k}\\
    (2\gamma_2\gamma_1)^n [G_2 u - u(0)] &= G_2 u^{(n)} - u^{(n)}(0) - 2\sum_{\text{odd }j=1}^{n-1}u^{(j)}(0)\cdot(2\gamma_2\gamma_1)^{n-j}[1].\label{CommuteG2 2k}
\end{align}

If $n=2k+1$ then
\begin{align}
    (2\gamma_2\gamma_1)^n G_1 u &=
    G_2u^{(n)} - u^{(n)}(0) - 2\sum_{\text{even }j=2}^{n-1}u^{(j)}(0)(2\gamma_2\gamma_1)^{n-j}[1],\label{CommuteG1 2k+1}\\
    (2\gamma_2\gamma_1)^n [G_2 u - u(0)] &=
    G_1u^{(n)}+ f \cdot u^{(n)}(0) - 2\sum_{\text{odd }j=1}^{n}u^{(j)}(0)(2\gamma_2\gamma_1)^{n-j}[1].\label{CommuteG2 2k+1}
\end{align}
\end{corollary}

\begin{proof}
Applying \eqref{G1G2commut} twice we obtain the following commutation relations valid on $C^2[-b,b]$,
\begin{equation}\label{G1G2commut2}
    (2\gamma_2\gamma_1)^2 G_1 = \bigl(G_1 + (f-2)\cdot \delta\bigr)\partial^2 \quad \text{and}\quad
    (2\gamma_2\gamma_1)^2 \bigl(G_2-\delta\bigr) = \bigl(G_2-\delta\bigr)\partial^2 - 4\gamma_2\gamma_1[1]\cdot\delta\partial.
\end{equation}
Now using relations \eqref{G1G2commut} and \eqref{G1G2commut2} the proof is straightforward by the induction. The inequalities $n\le p+4$ for the function $G_1u$ and $n\le p+3$ for the function $G_2u$ are justified by the following observation. The right-hand sides of the expressions \eqref{CommuteG1 2k}--\eqref{CommuteG2 2k+1} involve the functions $(\gamma_2\gamma_1)^j[1]$ for either $j\le n-2$ (function $G_1u$) or $j\le n-1$ (function $G_2u$). Note that we do not need to apply $(\gamma_2\gamma_1)^n$ to $1$ in the left-hand sides of \eqref{CommuteG2 2k} and \eqref{CommuteG2 2k+1}. According to Remark \ref{Rem gamma2gamma1 1}, the function $(\gamma_2\gamma_1)^j[1]$ is well defined for all $j\le p+2$. Hence the condition $n-2\le p+2$, i.e., $n\le p+4$, is sufficient for the function $G_1u$ and the condition $n-1\le p+2$, i.e., $n\le p+3$, is sufficient for the function $G_2u$.
\end{proof}

\begin{remark}
In the above corollary, the condition $q\in C^p[-b,b]$ is used only to ensure the existence of the functions $(2\gamma_2\gamma_1)^j[1]$, $j\le p+2$. The corollary holds whenever the corresponding number of the functions $(2\gamma_2\gamma_1)^j[1]$ are well-defined independently on the smoothness of $q$.
\end{remark}

Now we show that the result of Corollary \ref{Cor G1G2 uCn} can be inverted in the following sense. Supposing that for a function $\eta(x)$ the functions $(\gamma_2\gamma_1)^j[\eta]$, $j=1,2,\ldots$ are well defined, we are going to show that the preimages of the function $\eta$ under the operators $G_1$ and $G_2$ possess corresponding numbers of derivatives.

Let $u\in C^1[-b,b]$ be a given function and denote $\eta:=G_1u$. Applying $G_2^{-1}$ to the first relation in \eqref{G1G2commut} we obtain
\begin{equation*}
\begin{split}
  2G_2^{-1}\gamma_2\gamma_1\eta & = G_2^{-1}\bigl(G_2 u'-u'(0)\bigr)=u'-u'(0)G_2^{-1}[1] \\
    & = u'-u'(0) = \partial G_1^{-1}\eta - \partial G_1^{-1}\eta(0).
\end{split}
\end{equation*}
It follows from the definition of the operator $G_1$ that $\eta(0)=u(0)$ and $\eta'(0)=u'(0)+hu(0)$. Hence $\partial G_1^{-1}\eta(0)=u'(0)=\eta'(0)-h\eta(0)$ and we obtain the first inverse commutation relation
\begin{equation}\label{InverseCommut1}
    G_2^{-1}[2\gamma_2\gamma_1\eta] = \partial G_1^{-1}\eta - \eta'(0)+h\eta(0).
\end{equation}

In order to obtain the second commutation relation consider $u\in C_0^1[-b,b]$ and denote $\eta=G_2u$. Applying $G_1^{-1}$ to the second relation in \eqref{G1G2commut} we obtain
\begin{equation}\label{invcommut_a}
    2G_1^{-1}\gamma_2\gamma_1\eta = G_1^{-1}\bigl(G_1 u' + f\cdot u'(0)-2u'(0)\bigr)= u'+u'(0) - 2u'(0)G_1^{-1}[1],
\end{equation}
where we used the inverse of the equality $G_1[1]=f$. It follows from the definition of the operator $G_2$ that $\eta'(0)=u'(0)$, hence \eqref{invcommut_a} can be written as
\[
2G_1^{-1}\gamma_2\gamma_2\eta = \partial G_2^{-1}\eta + \eta'(0) - 2\eta'(0)G_1^{-1}[1],
\]
or
\begin{equation}\label{InverseCommut2}
    G_1^{-1}\bigl[2\gamma_2\gamma_1\eta + 2\eta'(0)\bigr] = \partial G_2^{-1}\eta+\eta'(0).
\end{equation}

Since the operator $G_1$ is a bijection on $C^1[-b,b]$, the commutation relation \eqref{InverseCommut1} holds on the whole $C^1[-b,b]$. Moreover, it means that for any $\eta\in C^1[-b,b]$ (which is equivalent to the condition that there exists $\gamma_2\gamma_1\eta$ and it is a continuous function), the preimage $G_1^{-1}\eta$ is a continuously differentiable function. Similarly, the commutation relation \eqref{InverseCommut2} holds on $C_0^1[-b,b]$ and allows one to conclude that the preimage $G_2^{-1}\eta$ is also a continuously differentiable function. Since $G_2[1]=1$, the commutation relation \eqref{InverseCommut2} can be extended onto $C^1[-b,b]$ in the following way
\begin{equation}\label{InverseCommut2a}
    G_1^{-1}\bigl[2\gamma_2\gamma_1(\eta-\eta(0)) + 2\eta'(0)\bigr] = \partial G_2^{-1}\eta+\eta'(0).
\end{equation}

Suppose now that for $\eta\in C_0^1[-b,b]$ there exists $(2\gamma_2\gamma_1)^2\eta$ and is a continuous function. Denote $v:=2\gamma_2\gamma_1\eta$. We obtain from \eqref{InverseCommut1} and \eqref{InverseCommut2} that $G_1^{-1}v$ is continuously differentiable and
\begin{equation*}
    G_1^{-1}v = G_1^{-1}(2\gamma_2\gamma_1\eta)=\partial G_2^{-1}\eta + \eta'(0)-2\eta'(0)\cdot G_1^{-1}[1].
\end{equation*}
If additionally $G_1^{-1}[1]$ is continuously differentiable (or equivalently, by Corollary \ref{Cor gamma2gamma1 commut}, there exists $2\gamma_2\gamma_1[1]$), then $G_2^{-1}\eta$ is twice continuously differentiable and
\begin{equation}\label{invcommut_b}
    G_2^{-1}\bigl((2\gamma_2\gamma_1)^2\eta\bigr) = \partial^2 G_2^{-1}\eta - 2\eta'(0)\cdot \partial G_1^{-1}[1]-(2\gamma_2\gamma_1\eta)'(0),
\end{equation}
where we have used that $\gamma_2\gamma_1\eta(0)=0$. The equality \eqref{InverseCommut1} for $\eta=1$ reads as
\begin{equation*}
    G_2^{-1}\bigl(2\gamma_2\gamma_1[1]\bigr) = \partial G_1^{-1}[1]+h,
\end{equation*}
hence we can write \eqref{invcommut_b} as
\begin{equation*}
    G_2^{-1}\bigl((2\gamma_2\gamma_1)^2\eta\bigr) = \partial^2 G_2^{-1} \eta - 2\eta'(0)\left(G_2^{-1}\bigl(2\gamma_2\gamma_1[1]\bigr)-h\right) - (2\gamma_2\gamma_1\eta)'(0),
\end{equation*}
or
\begin{equation}\label{InverseCommut22}
    G_2^{-1}\Bigl((2\gamma_2\gamma_1)^2\eta + 2\eta'(0)\cdot 2\gamma_2\gamma_1[1]\Bigr) = \partial^2 G_2^{-1}\eta + 2h \eta'(0) - (2\gamma_2\gamma_1\eta)'(0).
\end{equation}

Similarly we obtain from \eqref{InverseCommut2} that $G_1^{-1}\eta$ is twice continuously differentiable function and that
\begin{equation}\label{InverseCommut12}
    G_1^{-1}\bigl( (2\gamma_2\gamma_1)^2\eta + 2(2\gamma_2\gamma_1\eta)'(0)\bigr) = \partial^2 G_1^{-1}\eta + (2\gamma_2\gamma_1\eta)'(0).
\end{equation}

In order to obtain the higher order differentiability of the preimages consider the following functions
\begin{align}
    u_0 & = \eta & v_0 &= \eta-\eta(0) \label{u0v0}\\
    u_{2n+1} &= 2\gamma_2\gamma_1 [u_{2n}], & v_{2n+1} &= 2\gamma_2\gamma_1 [v_{2n}] + 2 v'_{2n}(0),\label{u1v1}\\
    u_{2n+2} &= 2\gamma_2\gamma_1 [u_{2n+1}] + 2 u'_{2n+1}(0), & v_{2n+2} &= 2\gamma_2\gamma_1 [v_{2n+1}], & n=0,1,\ldots\label{u2v2}
\end{align}
We suppose that the recursive process stops when we cannot evaluate the operator $\gamma_2\gamma_1$ anymore, i.e., the obtained function is not continuously differentiable. Note that the function $u_n$ is a linear combination of the functions $(2\gamma_2\gamma_1)^j \eta$, $j\le n$ and the functions $(2\gamma_2\gamma_1)^j[1]$, $j\le n-2$. The function $v_n$ is a linear combination of the functions $(2\gamma_2\gamma_1)^j \eta$, $j\le n$ and the functions $(2\gamma_2\gamma_1)^j[1]$, $j\le n$ (however if $\eta(0)=0$ the function $(2\gamma_2\gamma_1)^n[1]$ does not participate).

Starting from the relations \eqref{InverseCommut1}, \eqref{InverseCommut2}, \eqref{InverseCommut22} and \eqref{InverseCommut12} by induction we obtain the following result.

\begin{proposition}\label{Prop Derivative Preimage}
Let a function $\eta \in C^1[-b,b]$ be given and the functions $u_j$ and $v_j$ be well defined by the formulas \eqref{u0v0}--\eqref{u2v2} for $j\le n$. Then the preimages $G_1^{-1}\eta$ and $G_2^{-1}\eta$ are  $n$ times continuously differentiable functions and the following relations hold,
\begin{align}
    \partial^{n}G_1^{-1}\eta & = G_1^{-1}[u_{n}] - u'_{n-2}(0), & \partial^{n}G_2^{-1}\eta & = G_2^{-1}[v_n]-2h v'_{n-2}(0) + v'_{n-1}(0), &\text{if } n \text{ is even}, \label{InverseCommut1 2n}\\
    \partial^{n}G_2^{-1}\eta &=  G_1^{-1}[v_n] - v'_{n-1}(0), & \partial^{n}G_1^{-1}\eta & = G_2^{-1}[u_{n}] + u'_{n-1}(0)-h u_{n-1}(0),  & \text{if } n \text{ is odd}. \label{InverseCommut1 2n+1}
\end{align}
\end{proposition}

The existence of the functions $u_j$ and $v_j$ can be guaranteed by imposing smoothness conditions on $q$ and $\eta$, c.f., Remark \ref{Rem gamma2gamma1 1} and Corollary \ref{Cor G1G2 uCn}. The following corollary holds.
\begin{corollary}\label{Corr Derivative Preimage}
Suppose that $q\in C^p[-b,b]$ and let a function $\eta\in C^n[-b,b]$ be given. Then $G_1^{-1}\eta \in C^m[-b,b]$, where  $m=\min\{n, p+4\}$.

Let a function $\eta\in C^n_0[-b,b]$ be given. Then $G_2^{-1}\eta\in C^m_0[-b,b]$, where $m=\min\{n, p+3\}$.
\end{corollary}

The transmutation operators $G_1$ and $G_2$ together with Corollary \ref{Corr Derivative Preimage} allow us to prove a result similar to Proposition \ref{Prop a2n b2n} under different assumptions on the potential $q$. Suppose that in some neighborhood of $x=0$ the potential $q$ possesses continuous derivatives of all orders up to the order $N$. Then both functions $g_1(x)=\frac h2 + \frac 14\int_0^x q(s)\,ds$ and $g_2(x)=\frac 14\int_0^x q(s)\,ds$ participating in Theorem \ref{Th Kapprox} are $N+1$ times continuously differentiable at $x=0$ and there exist generalized derivatives $\gamma_1(\gamma_2\gamma_1)^j g_{1,2}$, $j\le N$ in a neighborhood of zero. By Theorem \ref{Thm Taylor} the following representations hold
\begin{equation}\label{Goursat conditions Taylor}
    \frac h2 + \frac 14\int_0^x q(s)\,ds = \sum_{n=0}^{N+1} a_n \frac{\mathbf{c}_n(x)}{n!}+o(x^{N+1})\quad \text{and}\quad
    \frac 14\int_0^x q(s)\,ds = \sum_{n=1}^{N+1} b_n \frac{\mathbf{s}_n(x)}{n!}+o(x^{N+1}).
\end{equation}
By Corollary \ref{Corr Derivative Preimage} the functions $\widetilde g_1:=G_1^{-1}g_1$ and $\widetilde g_2:=G_2^{-1}g_2$ are $N+1$ times continuously differentiable at $x=0$. Since the operators $G_1^{-1}$ and $G_2^{-1}$ are bounded and map the functions $\mathbf{c}_n$ and $\mathbf{s}_n$ into $2^{n-1}x^n$ (see Proposition \ref{Prop Operator G}), the following Taylor formulas are valid
\begin{equation*}
    \widetilde g_1(x)=a_0 + \sum_{n=1}^{N+1} \frac{2^{n-1}a_n x^n}{n!} + o(x^{N+1})\qquad\text{and}\qquad
    \widetilde g_2(x)=\sum_{n=1}^{N+1} \frac{2^{n-1}b_n x^n}{n!} + o(x^{N+1}).
\end{equation*}
Due to the definition of the operator $G$ the functions $\widetilde g_1$ and $\widetilde g_2$ coincide (up to some additive constants) with the functions $\phi(x)+\psi(x)$ and $\phi(x)-\psi(x)$ from \eqref{k via phi psi}, respectively. Repeating the proof of Proposition \ref{Prop a2n b2n} and using the uniqueness of the Taylor coefficients one can obtain the following result.
\begin{proposition}\label{Prop a2n b2n 2}
Let $q\in C^N[-c,c]$, where $[-c,c]\subset [-b,b]$ and let the coefficients $a_n$, $b_n$, $n\le N+1$ be defined by \eqref{Goursat conditions Taylor}. Then \[
a_{2n}=b_{2n},\qquad n=1,\ldots,\left[\frac{N+1}2\right].
\]
\end{proposition}

\section{Convergence rate estimates for the analytic approximation of transmutation operators}\label{Sect ConvEstimates}
In this section we establish relations between the smoothness of the potential $q$ and the decrease rate of $\varepsilon_{1,2}$ from Theorem \ref{Th Kapprox} as functions of $N$. As a result, we establish convergence rate estimates for the approximations $K_N$ of the integral kernel $\mathbf{K}$.

\begin{theorem}\label{Thm Direct}
Let $q\in C^p[-b,b]$. Then there exists a constant $\widetilde c_p$ and a sequence of the approximate kernels $K_N$ of the form \eqref{K(x,t)} such that
\begin{equation}\label{K-KN error}
    \max_{(x,t)\in\overline{\mathbf{S}}}\bigl| \mathbf{K}(x,t) - K_N(x,t)\bigr|\le \frac{\widetilde c_p}{N^{p+1}}.
\end{equation}
\end{theorem}
\begin{proof}
Consider the functions $g_1(x)=\frac h2 + \frac 14\int_0^x q(s)\,ds$ and $g_2(x)=\frac 14\int_0^x q(s)\,ds$ and their preimages $\widetilde g_1:=G_1^{-1}g_1$ and $\widetilde g_2:=G_2^{-1}g_2$. Since $g_1\in C^{p+1}[-b,b]$ and $g_2\in C_0^{p+1}[-b,b]$, by Corollary \ref{Corr Derivative Preimage} we have that $\widetilde g_1\in C^{p+1}[-b,b]$ and $\widetilde g_2\in C_0^{p+1}[-b,b]$. By Jackson's theorem \cite[Chap.4, Sec.6]{Cheney} there exist a constant $c_p$ and sequences of polynomials $\{p_N\}_{N\in\mathbb{N}}$ and $\{q_N\}_{N\in\mathbb{N}}$ such that
\begin{equation*}%\label{pNqN}
    \max_{x\in[-b,b]}\bigl| \widetilde g_1(x)-p_N(x)\bigr| \le \frac{c_p}{N^{p+1}}\qquad \text{and}\qquad \max_{x\in[-b,b]}\bigl| \widetilde g_2(x)-q_N(x)\bigr| \le \frac{c_p}{N^{p+1}},
\end{equation*}
and additionally polynomials $q_N$ satisfy $q_N(0)=0$.

Applying the operators $G_1$ and $G_2$ we obtain that for each $N$ there exist such complex numbers $a_0^{(N)},\ldots,a_N^{(N)},b_1^{(N)},\ldots,b_N^{(N)}$ that the constants $\varepsilon_1$ and $\varepsilon_2$ in \eqref{KxxErr} and \eqref{KxmxErr} satisfy
\[
\varepsilon_1\le \frac{c_p\|G_1\| }{N^{p+1}}\qquad \text{and}\qquad \varepsilon_2\le \frac{c_p\|G_2\| }{N^{p+1}},
\]
hence
\begin{equation*}
    \max_{(x,t)\in\overline{\mathbf{S}}}\bigl| \mathbf{K}(x,t) - K_N(x,t)\bigr|\le c(\varepsilon_1+\varepsilon_2)\le \frac{\tilde c_p}{N^{p+1}},
\end{equation*}
c.f., \cite[Theorem 5.1]{KT AnalyticApprox}.
\end{proof}

\begin{remark}
Let a function $g\in C^p[-b,b]$ be given. Then its sequence of polynomials of the best uniform approximation $\{p_n\}_{n\in\mathbb{N}}$ satisfies
\begin{equation*}
    \max_{x\in[-b,b]}\bigl| g(x)-p_n(x)\bigr|=o\left(\frac 1{N^{p+1}}\right),\qquad N\to \infty,
\end{equation*}
see, e.g., \cite[Thm. VIII]{Jackson}. Hence one can obtain a slightly stronger result than \eqref{K-KN error}, namely that
\begin{equation}\label{K-KN error2}
    \max_{(x,t)\in\overline{\mathbf{S}}}\bigl| \mathbf{K}(x,t) - K_N(x,t)\bigr| = o\left(\frac{1}{N^{p+1}}\right).
\end{equation}

For the estimate \eqref{K-KN error} to hold $q^{(p)}$ does not need to be a continuous function, it is sufficient that the derivative be bounded or that $q^{(p-1)}\in \operatorname{Lip} [-b,b]$. We do not enter into further details in the present paper.
\end{remark}

The order of the convergence rate in Theorem \ref{Thm Direct} is close to optimal as the following partial inverse result shows.  Before we need the following lemma.

\begin{lemma}\label{Lemma Inverse gamma2gamma1}
Suppose that for a function $u$ there exist all generalized derivatives $(2\gamma_2\gamma_1)^ju$ for $j\le k$ on a segment $[-d,d]\subset[-b,b]$, the last generalized derivative $(2\gamma_2\gamma_1)^ku\in C^\ell[-d,d]$ and $q\in C^m[-d,d]$, where $k$, $\ell$ and $m$ are some non-negative integers. Then the function $u$ possesses $\min\{k+\ell, m+2\}$ continuous derivatives on the segment $[-d,d]$.
\end{lemma}

\begin{proof}
The general solution of the equation $2\gamma_2\gamma_1 \eta=u$ can be written in the form
\begin{equation*}
    \eta(x)=c_1f(x) + c_2f(x)\int_0^x \frac {ds}{f^2(s)} + f(x) \int_0^x \biggl(\frac{u(s)}{f(s)} - \frac 1{f^2(s)}\int_0^s f'(t)u(t)\,dt\biggr)ds.
\end{equation*}
Hence if $u\in C^\ell$ and $f\in C^n$ then $\eta\in C^{\min\{n,\ell+1\}}$. The result of the lemma follows by induction taking into account that for $q\in C^m[-d,d]$ we have $f\in C^{m+2}[-d,d]$.
\end{proof}

\begin{theorem}\label{Thm Inverse}
Suppose that there exist a constant $p\in\mathbb{N}$, constants $\delta>0$ and $c_p>0$, and for each $N$ there exist such complex constants $a_0^{(N)},\ldots,a_N^{(N)},b_1^{(N)},\ldots,b_N^{(N)}$ that the constants $\varepsilon_1$ and $\varepsilon_2$ in \eqref{KxxErr} and \eqref{KxmxErr} satisfy
\begin{equation}\label{InvIneq}
    \max\{\varepsilon_1, \varepsilon_2\}\le \frac{c_p}{N^{p+1+\delta}}.
\end{equation}
Then $q\in C^p(-b,b)$.
\end{theorem}

\begin{proof}
Again, consider the functions $g_1(x)=\frac h2 + \frac 14\int_0^x q(s)\,ds$ and $g_2(x)=\frac 14\int_0^x q(s)\,ds$ and their preimages $\widetilde g_1:=G_1^{-1}g_1$ and $\widetilde g_2:=G_2^{-1}g_2$. Due to Proposition \ref{Prop Operator G} the preimages
\begin{equation*}
    G_1^{-1}\left(\sum_{n=0}^N a_n^{(N)} \mathbf{c}_n\right)\qquad \text{and}\qquad G_2^{-1}\left(\sum_{n=1}^N b_n^{(N)} \mathbf{s}_n\right)
\end{equation*}
are polynomials of degree less or equal to $N$. Since the operators $G_1^{-1}$ and $G_2^{-1}$ are bounded, the inequality \eqref{InvIneq} implies that the functions $\widetilde g_1$ and $\widetilde g_2$ can be approximated by sequences of polynomials in such a way that the approximation errors decrease at least as $O(N^{-p-1-\delta})$. Due to the inverse approximation theorem (see, e.g., \cite[Theorem 31]{KKTT}), $\widetilde g_1\in C^{p+1}(-b,b)\cap C^{[(p+1)/2]}[-b,b]$ and $\widetilde g_2\in C_0^{p+1}(-b,b)\cap C_0^{[(p+1)/2]}[-b,b]$.

Let $[-d,d]\subset(-b,b)$ and assume that $q\in C^m[-d,d]$ but $q\not\in C^{m+1}[-d,d]$ where $m<p$. In such case applying Corollary \ref{Cor G1G2 uCn} to the function $\widetilde g_1$ on the segment $[-d,d]$ we obtain that there exist generalized derivatives $(2\gamma_2\gamma_1)^j g_1$ of all orders up to $k=\min\{p+1, m+4\}$. Due to the assumption $m<p$ we have that $k\ge m+2$. Lemma \ref{Lemma Inverse gamma2gamma1} provides that $g_1\in C^{m+2}[-d,d]$, contradicting the assumption $q\not\in C^{m+1}[-d,d]$ as $q=g_1'$.
\end{proof}

\begin{remark}\label{Rem DirectInverse Alt}
There is another possibility to obtain convergence rate estimates similar to those of Theorems \ref{Thm Direct} and \ref{Thm Inverse}. It is based on the following simple observation: for each fixed $x$, the expression
\begin{equation}\label{KN alt}
    \widetilde K_N(x,t):= a_0 u_0(x,t) + \sum_{n=1}^{2[N/2]+1}a_n u_{2n-1}(x,t) + \sum_{n=1}^{2[(N+1)/2]} b_n u_{2n}(x,t),
\end{equation}
where $[\cdot]$ denotes the integer part function, is a polynomial in $t$ whose degree is less or equal to $N$. All generalized wave polynomials $u_n$ which are not included in \eqref{KN alt} contain terms with powers of $t$ higher than $t^N$, and expression \eqref{KN alt} contains exactly 2 terms with the power $t^N$, these terms are (up to nonzero multiplicative constants) $\varphi_0(x) t^N$ and $\varphi_1(x) t^N$. The $x$-derivative of \eqref{KN alt} possesses the same properties, the terms with $t^N$ are $\varphi'_0(x) t^N$ and $\varphi'_1(x) t^N$.

The mentioned properties are sufficient to verify that the Cauchy problem
\begin{equation}\label{Cauchy prob K}
\begin{split}
    \Bigl(\frac{\partial^2}{\partial x^2}  - q(x)\Bigr) K(x,t) = \frac{\partial^2}{\partial t^2}K(x,t),\\
    K(b,t) = p_1(t),\qquad
    K_1(b,t) = p_2(t),
\end{split}
\end{equation}
where $p_1$ and $p_2$ are some polynomials of degree not greater than $N$, possesses a unique solution in the region $0\le |t|\le x\le b$ and this solution has the form \eqref{KN alt}.

Now the analogue of Theorem \ref{Thm Direct} can be obtained if one approximates $K(b,t)$ and $K_1(b,t)$ by polynomials and utilizes continuous dependence (see, e.g., \cite{Trikomi}) of the Cauchy problem \eqref{Cauchy prob K} on the initial data. The analogue of Theorem \ref{Thm Inverse} can be obtained by utilizing continuous dependence of the solution of the Goursat problem \eqref{GoursatKh1}, \eqref{GoursatKh2} on the data on the characteristics and reasoning as in the proof of \cite[Theorem 3.5]{KNT}.

The described method does not require any inverse operators however the order of the convergence rate estimate which can be obtained is less by 1 than those provided by Theorem \ref{Thm Direct}. The reason is that for $q\in C^p[0,b]$ one has $K(b,t)\in C^{p+1}[-b,b]$ while $K_1(b,t) \in C^p[-b,b]$ and hence the convergence rate of the polynomial approximations of $K_1$ has estimates which are one order less than those for the function $K$.
\end{remark}

\section{Transmutation operators restricted to the half segment}\label{Sect 8}
For the solution of spectral problems it may be convenient to make use of the known initial values \eqref{ICcos} and \eqref{ICsin}. For that, one can consider equation \eqref{SLlambda} restricted to the segment $[0,b]$. The solutions $c(\omega, x)$ and $s(\omega, x)$ can be calculated once one knows the integral kernel $\mathbf{K}(x,t)$ on the domain $R_1=\{(x,t):0\le |t|\le x\le b\}$, and it can be deduced from the integral equation \eqref{IntEq for K} that the integral kernel $\mathbf{K}$ is uniquely determined on $R_1$ by the parameter $h$ and values of the potential $q$ on $[0,b]$ and does not depend on the values of $q$ for $x<0$. That is, let the potential $q$ be given on $[0,b]$. One can consider any sufficiently smooth continuation of $q$ onto the segment $[-b,b]$, the resulting integral kernel of the transmutation operator on $R_1$ does not depend on the continuation. However the analytic approximation of the transmutation operator given by Theorem \ref{Th Kapprox} does, as well as the inverse operators $\mathbf{T}_f^{-1}$, $G_{1,2}^{-1}$ and the preimage of the integral kernel $\mathbf{k}$.

Note that the functions $\varphi_k$, $\psi_k$, $\mathbf{c}_k$, $\mathbf{s}_k$, $k\in \mathbb{N}_0$, are completely determined on $[0,b]$ by the values on $[0,b]$ of a particular solution $f$ and instead of considering the approximation problems \eqref{KxxErr} and \eqref{KxmxErr} on the whole segment $[-b,b]$ one can solve them on the half-segment $[0,b]$. It is  clear that the coefficients $a_0,\ldots,a_N,b_1,\ldots,b_N$ determined as solutions of the approximation problems \eqref{KxxErr} and \eqref{KxmxErr} on the half-segment $[0,b]$ would not necessary solve them on the whole segment $[-b,b]$, it depends on the choice of the continuation of the potential $q$. We already verified in various numerical examples (see \cite{KT AnalyticApprox}) that the method based on restrictions of the approximation problems \eqref{KxxErr} and \eqref{KxmxErr} to the half-segment works. In this section we briefly justify the main results for this restricted case.

First we prove the main approximation theorem.
\begin{theorem}\label{Th Kapprox Restr} Let the complex numbers
$a_{0},\ldots,a_{N}$ and $b_{1},\ldots,b_{N}$ be such that the inequalities
\eqref{KxxErr} and \eqref{KxmxErr} hold
for any $x\in[0,b]$. Then the kernel $\mathbf{K}(x,t)$ is
approximated by the function \eqref{K(x,t)}
in such a way that for any $(x,t)\in R_1$ the inequality
holds
\begin{equation*}
\bigl|\mathbf{K}(x,t)-K_{N}(x,t)\bigr|\leq C(\varepsilon_1+\varepsilon_2),
\end{equation*}
where
\[
C=3 I_{0}\big(b\sqrt
{M}\big),\qquad M:=\max_{[0,b]}|q(x)|,
\]
and $I_{0}$ is the modified Bessel
function of the first kind.
\end{theorem}
\begin{proof}
Consider the difference $K^e = \mathbf{K} - K_N$ and the corresponding function $H^e(u,v)=K^e\big(u+v, u-v\big)$, $0\le u+v\le b$. It is the unique solution of the Goursat problem \eqref{GoursatTh1}, \eqref{GoursatTh2} and due to the inequalities \eqref{KxxErr} and \eqref{KxmxErr} the functions $H^e(u,0)$ and $H^e(0,v)$ are bounded by $\varepsilon_1+\varepsilon_2$. The Goursat problem \eqref{GoursatTh1}, \eqref{GoursatTh2} is equivalent to the following integral equation
\[
H^e(u,v) = H^e(u,0) + H^e(0,v) - H^e(0,0) + \int_0^u \int_0^v q(u'+v')H(u',v')\,du'\,dv',
\]
whose solution can be estimated by the successive iterations method, see, e.g., \cite[Proposition 18]{KKTT}. Note that $|H^e(0,0)| \le \varepsilon_1+\varepsilon_2$, hence the successive iterations provide that
\[
|H^e(u,v)|\le 3(\varepsilon_1+\varepsilon_2) \sum_{k=0}^\infty\frac{M^k |uv|^k}{k!k!}=3 (\varepsilon_1+\varepsilon_2)I_0\bigl(\sqrt{M|uv|}\bigr)\le
3 (\varepsilon_1+\varepsilon_2)I_0\bigl(b\sqrt{M}\bigr).
\]
\end{proof}

Approximations of the derivatives $c'(\omega,x)$ and $s'(\omega,x)$ require considering additionally a transmutation operator for the Darboux-associated equation
\begin{equation}\label{SL Darboux}
    -u''+q_D(x)u = \omega ^2 u,
\end{equation}
where
\begin{equation*}%\label{q Darboux}
    q_D = 2\left(\frac {f'}{f}\right)^2 - q.
\end{equation*}
The function $1/f$ is a particular solution of \eqref{SL Darboux} for $\omega=0$ satisfying $(1/f)(0)=1$ and $(1/f)'(0) = -h$. Consider the transmutation operator $\mathbf{T}_{1/f}$ for \eqref{SL Darboux}, where the subindex $1/f$ means that the operator $\mathbf{T}_{1/f}$ is such that $\mathbf{T}_{1/f}[1]=1/f$. Denote the integral kernel of $\mathbf{T}_{1/f}$ by $\mathbf{K}_D$. We refer the reader to \cite{KT AnalyticApprox}, \cite{KT Transmut} for further details.

As was shown in \cite{KT AnalyticApprox}, the pair of operators $\mathbf{T}_f$ and $\mathbf{T}_{1/f}$ allows one to write the derivatives of the solutions of \eqref{SLomega2} in the following form
\begin{align*}
    c'(\omega,x) & = \frac{f'}{f}c(\omega, x) - \omega \mathbf{T}_{1/f}[\sin\omega x],\\
    s'(\omega,x) & = \frac{f'}{f}s(\omega, x) + \mathbf{T}_{1/f}[\cos\omega x].
\end{align*}

Recall that the generalized wave polynomials for the operator $\mathbf{T}_{1/f}$ have the form
\begin{equation*}
v_{0}=\psi_{0}(x),\quad v_{2m-1}(x,t)=\sum_{\text{even }k=0}^{m}\binom
{m}{k}\psi_{m-k}(x)t^{k},\quad v_{2m}(x,t)=\sum_{\text{odd }k=1}^{m}%
\binom{m}{k}\psi_{m-k}(x)t^{k}.%\label{vm}%
\end{equation*}
Denote by $\mathbf{\widetilde c}_n$ and $\mathbf{\widetilde s}_n$ the corresponding traces, c.f., \eqref{cm} and \eqref{sm}.

It was proved in \cite[Theorem 6.4]{KT AnalyticApprox} that if one constructs an approximation of $\mathbf{K}$ via Theorem \ref{Th Kapprox} then the expression
\begin{equation}\label{KD_N}
    K_{D,N}= -b_{0}v_{0}(x,t)-\sum_{n=1}^{N}b_{n}v_{2n-1}(x,t)-\sum_{n=1}%
^{N}a_{n}v_{2n}(x,t),
\end{equation}
where $b_0:=a_0$, can be used as an approximation to the integral kernel $\mathbf{K}_D$. We present a simpler proof of this result, also suitable for the case of the half-segment.

\begin{theorem}\label{Th KDapprox}
Let the complex numbers
$a_{0},\ldots,a_{N}$ and $b_{1},\ldots,b_{N}$ be such that the inequalities
\eqref{KxxErr} and \eqref{KxmxErr} hold
for any $x\in[0,b]$. Then the kernel $\mathbf{K}_D(x,t)$ is
approximated by the function \eqref{KD_N}
in such a way that for any $(x,t)\in R_1$ the inequality
holds
\begin{equation}\label{K_D - K_D,N estimate}
\bigl|\mathbf{K}_D(x,t)-K_{D,N}(x,t)\bigr|\leq C\bigl(\varepsilon_1M_1+(\varepsilon_1+\varepsilon_2)(2M_1M_2b + M_1+1)\bigr),
\end{equation}
where
\[
C=3I_{0}\big(b\sqrt
{M}\big),\qquad M:=\max_{[0,b]}|q_D(x)|,\qquad M_1:=\max_{[0,b]}\left|\frac{1}{f(x)}\right|,\qquad M_2:=\max_{[0,b]}|f'(x)|
\]
and $I_{0}$ is the modified Bessel
function of the first kind.
\end{theorem}

\begin{proof}
Consider the functions
\begin{align}
    \widetilde g_1(x) &= -\frac h2 + \frac 14\int_0^x q_D(s)\,ds + \sum_{n=0}^N b_n \mathbf{\widetilde c}_n(x),\label{g1 def}\\
    \widetilde g_2(x) &= \frac 14\int_0^x q_D(s)\,ds + \sum_{n=1}^N a_n \mathbf{\widetilde s}_n(x).\label{g2 def}
\end{align}
Note that it is sufficient to show that the functions $\widetilde g_1$ and $\widetilde g_2$ are small enough to derive \eqref{K_D - K_D,N estimate} from Theorem \ref{Th Kapprox Restr} applied to equation \eqref{SL Darboux}. Consider also the functions
\begin{align*}
    g_1(x) &= \frac h2 + \frac 14\int_0^x q(s)\,ds - \sum_{n=0}^N a_n \mathbf{c}_n(x),\\
    g_2(x) &= \frac 14\int_0^x q(s)\,ds - \sum_{n=1}^N b_n \mathbf{s}_n(x),
\end{align*}
the left-hand sides of the inequalities \eqref{KxxErr} and \eqref{KxmxErr}.

One can easily verify that the following equalities hold
\begin{gather*}
    \frac 1f \partial f\biggl( \frac 14 \int_0^x q_D(s)\,ds\biggr)  = - f\partial \frac 1f \biggl(\frac h2 +\frac 14\int_0^x q(s)\,ds\biggr),\\
    \frac 1f \partial f\biggl( -\frac h2+\frac 14 \int_0^x q_D(s)\,ds\biggr)  = - f\partial \frac 1f \biggl(\frac 14\int_0^x q(s)\,ds\biggr),\\
    f\partial\frac 1f \mathbf{c}_m = m(\mathbf{s}_{m-1} + \mathbf{\widetilde c}_{m-1}) = \frac 1f \partial f \mathbf{\widetilde s}_m,\qquad m=1,\ldots\\
    f\partial\frac 1f \mathbf{s}_m = m(\mathbf{c}_{m-1} + \mathbf{\widetilde s}_{m-1}) = \frac 1f \partial f \mathbf{\widetilde c}_m,\qquad m=1,\ldots\\
    f\partial \frac 1f \mathbf{c}_0 = \frac 1f \partial f \mathbf{\widetilde c}_0=0,
\end{gather*}
where for convenience we denoted $\mathbf{s}_0=\mathbf{\widetilde s}_0=0$. Hence
\begin{equation}\label{g1 Goursat}
    f\partial \frac 1f g_1 = -\frac 1f \partial f \widetilde g_2
\end{equation}
and
\begin{equation}\label{g2 Goursat}
    f\partial \frac 1f g_2 = -\frac 1f \partial f \widetilde g_1.
\end{equation}
Inverting the equalities \eqref{g1 Goursat} and \eqref{g2 Goursat} one obtains that
\begin{equation}\label{g1 inv}
    \begin{split}
       \widetilde g_1(x) &=  \frac{c_1}{f(x)} - \frac 1{f(x)}\int_0^x f^2(s)\left(\frac{g_2(s)}{f(s)}\right)^{\prime}\,ds  \\
         & = \frac{c_1}{f(x)} - g_2(x) + \frac{g_2(0)}{f(x)} + \frac 2{f(x)}\int_0^x f'(s)g_2(s)\,ds
     \end{split}
\end{equation}
and
\begin{equation}\label{g2 inv}
       \widetilde g_2(x) = \frac{c_2}{f(x)} - g_1(x) + \frac{g_1(0)}{f(x)} + \frac 2{f(x)}\int_0^x f'(s)g_1(s)\,ds,
\end{equation}
where $c_1$ and $c_2$ are some constants. Substituting $x=0$ into \eqref{g1 inv}, \eqref{g2 inv} and \eqref{g1 def}, \eqref{g2 def} one can find that $c_1 = b_0 - \frac h2=a_0-\frac h2 = -g_1(0)$ and $c_2=0$. Finally, using the inequalities $|g_1(x)|\le \varepsilon_1$ and $|g_2(x)|\le \varepsilon_2$, $x\in [0,b]$, one obtains the following estimates
\begin{equation*}
    \left|\widetilde g_1(x)\right|\le \varepsilon_1 M_1 + \varepsilon_2(2M_1 M_2|x| + M_1 + 1)\qquad \text{and}\qquad \left|\widetilde g_2(x)\right|\le \varepsilon_1(2M_1 M_2|x| + M_1 + 1),
\end{equation*}
which combined with Theorem \ref{Th Kapprox Restr} finish the proof.
\end{proof}

\end{document}